\documentclass{article}
\usepackage{pifont}
\usepackage{bbding}
\usepackage{amssymb}
\usepackage{amsmath}
\usepackage{latexsym}
\usepackage{amssymb}
\usepackage{graphicx}
\usepackage{amssymb,amsmath,array,epsf,epsfig}
\usepackage{mathdots}
\usepackage{cite}
\usepackage{bm}
\usepackage{multirow}
\usepackage{extarrows}
\usepackage{xcolor}%
\usepackage{textcomp}%
\usepackage{manyfoot}%

\newtheorem{lemma}{Lemma}
\newtheorem{theorem}{Theorem}
\newtheorem{proposition}{Proposition}

\newenvironment{proof}{\noindent{\bf Proof:}}{\hfill\fbox{}\vspace*{1mm}}
\numberwithin{equation}{section}
\numberwithin{lemma}{section}
\numberwithin{theorem}{section}

\begin{document}
\title{A novel fourth-order scheme for two-dimensional Riesz space fractional nonlinear reaction-diffusion equations and its optimal preconditioned solver
\thanks{The work of Wei Qu was supported by the research grants 2020A1515110454 from Guangdong Basic and Applied Basic Research Foundation, and 2021KCXTD052 from the Scientific Computing Research Innovation Team of Guangdong Province. The work of Sean Hon was supported in part by the Hong Kong RGC under grant 22300921 and a start-up grant from the Croucher Foundation. The work of Siu-Long Lei was supported by the research grants MYRG2022-00262-FST and MYRG-GRG2023-00181-FST-UMDF from University of Macau.
}}
\setcounter{footnote}{-1}
\author{
Wei Qu$^{1}$,\, Yuan-Yuan Huang$^{2}$,\, Sean Hon$^{2}$,\, Siu-Long Lei$^{3}$\footnote{Correspondence: {\tt sllei@um.edu.mo}}
\\
\\
1.School of Mathematics and Statistics, Shaoguan University, Shaoguan, China.\\
2.Department of Mathematics, Hong Kong Baptist University, Kowloon Tong, \,\,\\
Hong Kong SAR, China.\\
3.Department of Mathematics, University of Macau, Macau SAR, China.\,\,\,\,\,\,\,\,\,\,\,\,\,\,\,\,\,\,}

\date{}
\maketitle

\begin{abstract}
A novel fourth-order finite difference formula coupling the Crank-Nicolson explicit linearized method is proposed to solve Riesz space fractional nonlinear reaction-diffusion equations in two dimensions. Theoretically, 
under the Lipschitz assumption on the nonlinear term, 
%under the assumption for the nonlinear term satisfying the Lipschitz condition, 
%if the nonlinear term satisfies the Lipschitz condition, 
the proposed high-order scheme is proved to be unconditionally stable and convergent in the discrete $L_2$-norm. Moreover, a $\tau$-matrix based preconditioner is developed to speed up the convergence of the conjugate gradient method with an optimal convergence rate (a convergence rate independent of mesh sizes) for solving the symmetric discrete linear system. Theoretical analysis shows that the spectra of the preconditioned matrices are uniformly bounded in the open interval $(3/8,2)$. To the best of our knowledge, this is the first attempt to develop a preconditioned iterative solver with a mesh-independent convergence rate for the linearized high-order scheme. Numerical examples are given to validate the accuracy of the scheme and the effectiveness of the proposed preconditioned solver.
\end{abstract}


\bigskip


\section{Introduction}
In this article, we focus on the following two-dimensional (2D) Riesz space fractional nonlinear reaction-diffusion equations (RSFNRDEs):
\begin{eqnarray}\label{2D-RSFNRDEs-eq}
\begin{array}{l}
\left\{
  \begin{array}{ll}
    \displaystyle{\frac{\partial u(x,y,t)}{\partial t}=K_{\alpha}\frac{\partial^\alpha u(x,y,t)}{\partial |x|^\alpha}+K_\beta\frac{\partial^\beta u(x,y,t)}{\partial |y|^\beta}+f(x,y,t,u(x,y,t))}, \quad (x,y,t)\in\Omega\times(0,T]\\
    \vspace{2mm}
    u(x,y,0)=u_0(x,y),\quad (x,y)\in\Omega,\\
    %\vspace{2mm}
%    u(x,y,t)=0,\quad (x,y,t)\in\partial\Omega\times(0,T] , 
    \vspace{2mm}
    u(x,y,t)=0,\quad (x,y,t)\in({\mathbb{R}^2}\textbackslash\Omega)\times(0,T] ,
  \end{array}
\right.
\end{array}
\end{eqnarray}
where $\Omega=(x_L,x_R)\times(y_D,y_{\small U})$, $\alpha,\beta\in(1,2)$, 
%The diffusion coefficients 
$K_\alpha, K_\beta>0$ are two constants, 
and  
%$\partial\Omega$ denotes the boundary of $\Omega$, 
$u_{0}(x,y)$ is the initial condition. $f(x,y,t,u(x,y,t))$ is a nonlinear and continuous source term, and satisfies the Lipschitz condition, i.e.,
\begin{equation}\label{Lipschitz-condition}
\vert f(x,y,t,u)-f(x,y,t,v)\vert\leq L_1|u-v|,
\end{equation}
where $L_1$ is a positive constant. In particular, when $f(x,y,t,u)=f(x,y,t)$, Eq. (\ref{2D-RSFNRDEs-eq}) reduces to a linear problem. The Riesz fractional derivatives $\frac{\partial^\alpha u(x,y,t)}{\partial |x|^\alpha}$ and $\frac{\partial^\beta u(x,y,t)}{\partial |y|^\beta}$ are proportional to the average of left and right Riemann-Liouville (R-L) fractional derivatives,
%concerning to $1<\alpha,\beta<2$ 
which are defined as follows:
$$
\frac{\partial^{\alpha} u(x,y,t)}{\partial |x|^{\alpha}}
=\sigma_{\alpha}\Big({_{x_L}}\mathcal {D}_{x}^{\alpha}u(x,y,t) +{_{x}}\mathcal {D}_{x_R}^{\alpha}u(x,y,t)\Big),~~\sigma_{\alpha}=-\frac{1}{2\cos(\frac{\alpha\pi}{2})}>0,
$$
$$
\frac{\partial^{\alpha} u(x,y,t)}{\partial |y|^{\beta}}
=\sigma_{\beta}\Big({_{y_D}}\mathcal {D}_{y}^{\beta}u(x,y,t) +{_{y}}\mathcal {D}_{y_U}^{\beta}u(x,y,t)\Big),~~\sigma_{\beta}=-\frac{1}{2\cos(\frac{\beta\pi}{2})}>0,
$$
where ${_{x_L}}\mathcal {D}_{x}^{\alpha}u(x,y,t)$, ${_{x}}\mathcal {D}_{x_R}^{\alpha}u(x,y,t)$, ${_{y_D}}\mathcal {D}_{y}^{\beta}u(x,y,t)$, and ${_{y}}\mathcal {D}_{y_U}^{\beta}u(x,y,t)$ are left and right R-L fractional derivatives of orders $1<\alpha,\beta<2$, respectively. For the definitions of left and right R-L fractional derivatives, we refer readers to \cite{MT1,MT2} for details. 
%More precisely,
%$$
%\begin{array}{l}
%\displaystyle{\,{_{x_L}}\mathcal {D}_x^{{\alpha}}u(x,y,t)=\dfrac{1}{\Gamma(2-\alpha)}\dfrac{\partial^2}{\partial
%x^2}\int_{x_L}^{x} \dfrac{u(\xi,y,t)}{(x-\xi)^{\alpha-1}}\,\text{d}\xi},
%\displaystyle{\,{_x}\mathcal {D}_{x_R}^{{\alpha}}u(x,y,t)=\dfrac{1}{\Gamma(2-\alpha)}\dfrac{\partial^2}{\partial
%x^2}\int_{x}^{x_R} \dfrac{u(\xi,y,t)}{(\xi-x)^{\alpha-1}}\,\text{d}\xi},\\
%\displaystyle{\,{_{y_D}}\mathcal {D}_y^{{\beta}}u(x,y,t)=\dfrac{1}{\Gamma(2-\beta)}\dfrac{\partial^2}{\partial
%y^2}\int_{y_D}^y \dfrac{u(x,\eta,t)}{(y-\eta)^{\beta-1}}\,\text{d}\eta},
%\displaystyle{\,{_y}\mathcal {D}_{y_U}^{{\beta}}u(x,y,t)=\dfrac{1}{\Gamma(2-\beta)}\dfrac{\partial^2}{\partial
%y^2}\int_y^{y_U} \dfrac{u(x,\eta,t)}{(\eta-y)^{\beta-1}}\,\text{d}\eta},\\
%\end{array}
%$$
%with $\Gamma(\cdot)$ being the gamma function.

%Over the past few decades, fractional differential equations (FDEs) have attracted growing attention, due to the fact that they 
Fractional differential equations (FDEs) serve as powerful mathematical tools for accurate descriptions of challenging phenomena, such as biology \cite{Magin1}, physics \cite{BWM1,BWM2,CLZ1}, random walks \cite{BMK1}, image processing \cite{BF1}, finance \cite{raberto2002waiting} and control problems \cite{du2016fast,pougkakiotis2020fast}. With the increase in applications involving FDEs, the study of the solution to FDEs has been significant. However, since analytical solutions for FDEs are rarely obtainable, these types of equations are of numerical interest, which also makes
%the
%investigation of numerical methods for
the computation of numerical solutions an active field of research. In the realm of the numerical method for FDEs, there have been significant and successful developments, including
%some important and successful attempts are appeared, e.g.,
the finite difference method \cite{MT1,MT2,SL1,ZTD1,LNS3,TZD1,LL1,LSH1,she2022class,DL1,
ding2015high,ding2023high,HSC1,CD2} the finite element method \cite{JLPZ1,zhao2015finite,BTY,D1} and the finite volume method \cite{donatelli2018spectral,fu2019stability,LZTBA1,liu2020analysis}. Among those methods,
%Along this line,
it is not difficult to find that getting numerical solutions with high-order accuracy is the most necessary and intuitive way.

In what follows, we will first review some progress (discretization and algorithms) on the space FDEs with Riesz fractional derivatives like RSFNRDEs (\ref{2D-RSFNRDEs-eq}). So far, for discretizing the Riesz fractional derivative, there are two common methods: One is the indirect method, which is based on the relationship between the left and right R-L fractional derivatives and the Riesz fractional derivative; hence, the discretization methods with first-order to fourth-order accuracies for the Riesz fractional derivative can be founded in \cite{MT1,MT2,SL1,TZD1,LL1,LSH1,she2022class,HSC1}. The other one is the direct method, which is especially for the Riesz fractional derivative. In more details, Ortigueira \cite{ortigueira2006riesz} established a second-order formula for the Riesz fractional derivative. Based on the work in \cite{ortigueira2006riesz}, \c{C}elik and Duman \cite{CD2} proposed a second-order fractional centered difference (FCD) discretization for solving fractional diffusion equations. Subsequently, Xiao and Wang \cite{xiao2019symplectic} constructed a fourth-order FCD formula for the Riesz fractional derivative and successfully applied it to the fractional Schr\"odinger equation. Recently, according to a new appropriate generating function resulted from the modifications and improvements of \cite{CD2,ortigueira2006riesz}, Ding and Li \cite{ding2023high} established a novel fourth-order formula for the Riesz fractional derivative. As already
observed in \cite{ding2023high}, after discretized
%discretization
by the novel fourth-order formula to the space fractional complex Ginzburg-Landau equation, the high-order scheme has the same structure as the low-order one, which is suitable for the extension to other FDEs. In addition, the study of some other high-order formulas to the Riesz fractional derivative can be found in \cite{ZTD1,ding2015high,ding2016high,ding2017high,cheng2019novel,hu2020fourth}.

Note that the fractional differential operator is nonlocal, thus numerical solutions of FDEs are often required to solve dense discrete linear systems with a Toeplitz-like structure \cite{WWS1}. By exploiting this matrix structure, the discrete linear systems can be solved via fast Fourier transforms (FFTs).
%and fast discrete sine transform (DST).
However, as usual, the systems are often ill-conditioned as the matrix size increases, which means that solving such a linear system is computationally intensive. To reduce the computational cost, fast algorithms based on FFTs coupling preconditioning techniques have been proposed extensively. Among them, we mention some popular preconditioners including circulant preconditioner \cite{lei2013circulant,lei2016multilevel}, approximate inverse preconditioner \cite{PKNS1}, structure-preserving preconditioner \cite{DMS1}, banded preconditioner \cite{LYJ1,SLYL1,JLZ1}, multigrid preconditioner \cite{moghaderi2017spectral}, and splitting preconditioner \cite{LNS1,BLP1}. Note that the convergence rates of the iterative methods with the above preconditioners are not mesh-independent in general, which can lead to an increase in the number of iterations when the size of the matrix is increasing. It is remarked that this phenomenon is clearly evident in the multi-dimensional setting.
%Therefore, it is expected and interesting to investigate the mesh-sizes independent convergence rate for some preconditioned iterative method.
Therefore, investigating the mesh-independent convergence rate for some preconditioned iterative methods constitutes an interesting and promising research direction.
%field of research.
Along this line, circulant-type preconditioners are not a good choice since most of the preconditioned matrices are merely expressed as
the sum of the identity matrix plus a small norm matrix and a low-rank matrix, which cannot achieve either mesh-independent or superlinear convergence rate in the multi-dimensional cases, see \cite{capizzano2000any} for more details.
%circulant-like preconditioners will not taken into account since its negative results regarding the convergence rate of preconditioned iterative method in a high-dimension, see \cite{capizzano2000any}.

Recently, to solve the symmetric Toeplitz discrete linear system resulted from some FDEs with Riesz fractional derivative, the $\tau$ preconditioner has gained significant interest and success, which is shown in its superior performance compared with the circulant version. Since such preconditioner can be diagonalized by fast discrete sine transforms (DSTs), allowing us to perform the inversion and the matrix-vector multiplication involving $\tau$ matrix in ${\mathcal O}(N\log N)$ operations, where $N$ is the matrix size.
%number of interior mesh points in space.
%is the number of grid points.
%then it can be inverted in ${\mathcal O}(N\log N)$ operations by the fast DST.
%%like circulant preconditioner.
%Moreover, the complexity of the matrix-vector multiplication involving $\tau$ matrix is also reduced to ${\mathcal O}(N\log N)$ thanks to the use of the fast DST.
In addition, it costs ${\mathcal O}(N)$ storage requirement for the $\tau$ matrix by only storing its first column. For these reasons,
%With these advantages,
$\tau$ preconditioner is indeed suitable for solving the symmetric Toeplitz system. Specifically, in \cite{huang2021spectral}, Huang et al. proposed the $\tau$ preconditioner for solving symmetric and positive definite (SPD) Toeplitz linear system resulted from Riesz fractional diffusion equations in high dimensions. More importantly, they have shown that all eigenvalues of the preconditioned matrices lie in $(1/2,3/2)$,
%are uniformly bounded in the open interval $(1/2,3/2)$, 
which ensures the mesh-independent convergence rate for the preconditioned conjugate gradient (PCG) method. After that,
%motivated by this idea,
the authors also extended the $\tau$ preconditioner for solving Riesz space fractional diffusion equations in convex domains \cite{huang2021preconditioner}, multi-dimensional fractional diffusion equations with R-L fractional derivatives \cite{lin2023tau}, and space fractional Cahn-Hilliard equations \cite{huang2022preconditioners,huang2023preconditioned}. Motivated by the work \cite{huang2021spectral}, a $\tau$ preconditioner is designed by Zhang et al. \cite{zhang2022fast} to efficiently solve the discrete linear systems stemming from the fractional centered discretization of the considered 2D RSFNRDEs (\ref{2D-RSFNRDEs-eq}) and 
%the lower and upper bounds of 
the spectra of the preconditioned matrices are also shown to be independent of mesh sizes and far from zero. %are also discussed.
%In Reference \citenum{barakitis2022preconditioners}, 
Moreover, according to the spectral symbol,
%when the generating functions of coefficient matrices have a unique zero at zero of real positive order between one and two,
Barakitis et al. \cite{barakitis2022preconditioners} proposed a novel $\tau$-algebra preconditioner for ill-conditioned symmetric Toeplitz discrete linear systems where the generating functions of coefficient matrices have a unique zero at 
%zero of fractional order between one and two,
the origin of fractional order ranging from $1$ to $2$,
which outperforms
%performs better than
the preconditioners proposed in \cite{DMS1,moghaderi2017spectral}. For some other applications of $\tau$-matrix based preconditioner, we refer the reader to \cite{tang2024new,tang2022lopsided,tang2023matrix,huang2023tau,zeng2022tau,mazza2023algebra,
aceto2023rational,shao2022preconditioner,LinDongHon2023,hon2023sine} and references therein. It is worth mentioning that developing a preconditioning strategy based on the fast DSTs to solve the high-order symmetric Toeplitz linear systems resulted from the space FDEs with Riesz fractional derivative is needed. To the best of our knowledge, there is no preconditioned iterative method with $\tau$-matrix based preconditioner for solving such kind of linear systems that has a mesh-independent convergence rate, which inspires us to consider this problem.

The main objective of this paper consists of two aspects: the first is to develop a novel fourth-order scheme for the 2D RSFNRDEs that is unconditionally stable and convergent.
%the first is to develop a novel unconditional stable and convergent fourth-order scheme for the the 2D RSFNRDEs (\ref{2D-RSFNRDEs-eq}).
The second is to propose an efficient preconditioner for the symmetric two-level Toeplitz linear systems resulted from the 2D RSFNRDEs (\ref{2D-RSFNRDEs-eq}) so that the PCG method with the proposed preconditioner converges independently of the discretization mesh sizes. Different from the existing preconditioning strategy mentioned in \cite{huang2023tau}, in this article, we show that the spectra of the preconditioned matrix fall in an open interval for all $1<\alpha,\beta<2$, as opposed to
\cite{huang2023tau} where the spectra are only
%unlike some of them are
verified through numerical experiments for $1.4674<\alpha,\beta<2$. In addition, it is the first preconditioned iterative solver in the literature with a mesh-independent convergence rate for the linearized high-order scheme. The numerical results also emphasize that the proposed preconditioner needs fewer iterations and CPU times than circulant-type preconditioners, such as T. Chan's and Strang's circulant preconditioners \cite{Jin1,chan1989toeplitz,chan2007introduction}.

%this is in line with the
%theoretical results

The remainder of this article is arranged as follows. Section 2 reviews some preliminaries on the novel fourth-order approximation to the Riesz fractional derivative. In Section 3, we develop a high-order fully discrete numerical scheme for 2D RSFNRDEs, and its unconditional stability and convergence analysis are also provided under the discrete $L_2$-norm. In Section 4, we design the preconditioning strategies for the discrete linear system and study the spectra of the preconditioned matrix. In Section 5, some numerical results are reported to show the correctness of the scheme and the effectiveness of the proposed preconditioner. Some conclusions of this article are drawn in Section 6.

\section{A fourth-order approximation to the Riesz fractional derivative}
In \cite{ding2023high}, Ding introduced a novel fourth-order approximation to the Riesz fractional derivative, which is the modification and improvement of the approximations given in \cite{CD2}. The main results are listed below.

\begin{lemma}{\rm (\cite{ding2023high})}
Let 
\begin{equation}\nonumber
\mathcal{L}^{4+\mu}(\mathbb{R})=\left\{u(x)\in L^{1}(\mathbb{R})|\int_{-\infty}^{\infty}(1+|\eta|)^{4+\mu}|\hat{u}(\eta)|d\eta<\infty\right\},
\end{equation}
be the fractional Sobolev space.
%where $\hat{u}(\eta)=\int_{-\infty}^{\infty}u(x)e^{{\bf i}\eta x}dx$ is the Fourier transformation of $u(x)$ for all $\eta\in(-\infty,\infty)$ and ${\bf i}=\sqrt{-1}$. 
Suppose $u(x)\in\mathcal{L}^{4+\mu}(\mathbb{R})$, then we have the following fourth-order approximation:
%the Riesz fractional derivatives 
%$\frac{\partial^\alpha u(x,t)}{\partial |x|^\alpha}$%
%can be approximated as
%\begin{eqnarray}
%\dfrac{\partial^\alpha u(x)}{\partial |x|^\alpha}=\delta_{x}^{\alpha}u(x)+\mathcal{O}(h^4), \quad h\to 0,
%\end{eqnarray}
\begin{eqnarray}\nonumber
\dfrac{\partial^\alpha u(x)}{\partial |x|^\alpha}=-\frac{1}{h^\alpha}\sum_{k=-\infty}^{\infty}s_{k}^{(\alpha)}
u(x-kh)+\mathcal{O}(h^4), \quad h\to 0,
\end{eqnarray}
where %$h$ is the spatial step size.
%Then the fourth-order operator is defined by
%\begin{eqnarray}\label{4FCD}
%\delta_{x}^{\gamma}u(x)=-\frac{1}{h^\alpha}\sum_{k=-\infty}^{\infty}g_{k}^{(\alpha)}
%u(x-kh),
%\end{eqnarray}
%Here 
the coefficients $s_{k}^{(\alpha)}$ $(k=0,\pm 1,\pm 2,\ldots)$ are determined by the Fourier expansion of the generating function
%$S(\omega)$, defined in Eq.(\ref{generatingS}) i.e.,
\begin{equation}\label{generatingS}
S(\omega):=\left[1+\frac{\alpha}{24}(2-\omega-\omega^{-1})\right](2-\omega-\omega^{-1})^\frac{\alpha}{2}=\sum_{k=-\infty}^{\infty}s_{k}^{(\alpha)}\omega^k,\quad |\omega|\leq 1.
\end{equation}
Setting $\omega=e^{{\bf i}\theta}$ with ${\bf i}=\sqrt{-1}$ in Eq. (\ref{generatingS}) and using the inverse Fourier transform formula coupling the following formula,
\begin{equation}\label{Riesz_coef}
\frac{1}{2\pi}\int_{-\pi}^{\pi}(2-e^{{\bf i}\theta}-e^{{\bf -i}\theta})^{\frac{\alpha}{2}e^{-{\bf i} k\theta}}d\theta=\frac{(-1)^{k}\Gamma(\alpha+1)}{\Gamma(\frac{\alpha}{2}-k+1)\Gamma(\frac{\alpha}{2}+k+1)},
\end{equation}
one gets the explicit expression of the coefficient $s_{k}^{(\alpha)}$,
\begin{equation}\label{coesk}
\begin{footnotesize}
s_{k}^{(\alpha)}=\frac{(-1)^{k}\Gamma(\alpha+1)}{\Gamma(\frac{\alpha}{2}-k+1)\Gamma(\frac{\alpha}{2}+k+1)}
\left[1+\frac{\alpha(\alpha+1)(\alpha+2)}{6(\alpha-2k+2)(\alpha+2k+2)}\right],\quad k=0,\pm 1,\pm 2, \ldots.
\end{footnotesize}
\end{equation}
\end{lemma}
It is noted that the coefficients $s_{k}^{(\alpha)}$ satisfy the properties as follows.
\begin{proposition}{\rm (\cite{ding2023high})}\label{coef_prop}
For all $1<\alpha<2$, and let $s_{k}^{(\alpha)}$ be defined as (\ref{coesk}). Then we have
\begin{description}
  \item[(i)]  $s_{k}^{(\alpha)}=s_{-k}^{(\alpha)}$ \mbox{for all} $k\geq 1$;
  \item[(ii)]  $s_{0}^{(\alpha)}>0$,\,$s_{\pm 1}^{(\alpha)}<0$;
   \item[(iii)]
   $s_{\pm 2}^{(\alpha)}\left\{
                               \begin{array}{ll}
                                 <0, \,& \hbox{$\alpha\in(1,\alpha^{*})$,} \\
                                 \geq 0, \,& \hbox{$\alpha\in[\alpha^{*},2)$,}
                               \end{array}
                             \right.$
   where $\alpha_{*}\approx 1.6516$;
   %$s_{2}^{(\gamma)}\leq 0$ if $\alpha\in(1,\alpha^{*}]$, while $s_{2}^{(\gamma)}> 0$ if $\alpha\in(\alpha^{*},2)$\, where $\gamma_{*}\approx 1.6516$.
    \item[(iv)] $s_{\pm k}^{(\alpha)}<0$, $k=3, 4,\ldots$; \mbox{and}
    \item[(v)]  $\sum\limits_{k=-\infty}^{\infty}s_{k}^{(\alpha)}=0.$
\end{description}
\end{proposition}

\section{Fully discrete numerical scheme and its theoretical analysis}
In this section, we develop the fully discrete numerical scheme for 2D RSFNRDEs. Let $\Delta t=\frac{T}{M}$, $h_x=\frac{x_R-x_L}{N_x+1}$ and $h_y=\frac{y_U-y_D}{N_y+1}$ with $M$, $N_x$ and $N_y$ being three positive integers. By adopting the fourth-order approximation given in Section 2, we discretize the Riesz fractional derivatives at point $(x_i,y_j,t_m)$ as
\begin{eqnarray}\label{4th-discretization}
\begin{array}{l}
\left\{
  \begin{array}{ll}
    \displaystyle{\frac{\partial^\alpha u(x_i,y_j,t_m)}{\partial |x|^\alpha}=-\frac{1}{h_x^\alpha}\sum_{k=-N_x+i}^{i-1}s_{k}^{(\alpha)}
     u(x_{i-k},y_j,t_m)+\mathcal{O}(h_x^4)}, \\
    \vspace{2mm}
    \displaystyle{\frac{\partial^\beta u(x_i,y_j,t_m)}{\partial |y|^\beta}=-\frac{1}{h_y^\beta}\sum_{k=-N_y+j}^{j-1}s_{k}^{(\beta)}
     u(x_{i},y_{j-k},t_m)+\mathcal{O}(h_y^4)}, \\
  \end{array}
\right.
\end{array}
\end{eqnarray}
where $t_m=m\Delta t$ for $m=0,1,\ldots,M$, 
%$h_x=\frac{x_R-x_L}{N_x+1}$ and $h_y=\frac{y_U-y_D}{N_y+1}$ with
%%are the spatial step sizes in $x$-direction and $y$-direction, respectively, where 
%$N_x$ and $N_y$ being two positive integers. Then 
%the spatial and temporal partition are 
%$t_m=m\Delta t$ for $m=0,1,\ldots,M$, 
$x_i=x_L+ih_x$ for $i=0,1,\ldots,N_x+1$, and $y_j=y_D+jh_y$ for $j=0,1,\ldots,N_y+1$. 
%To construct a linearized high-order scheme for 2D RSFNRDEs, 
Then we use the Taylor expansion to approximate the first-order time partial derivative at point $(x_i,y_j,t_{m+1/2})$
%also apply the Crank-Nicolson (CN) scheme to approximate the first-order time partial derivative
\begin{equation}\label{CN}
\frac{\partial u(x_i,y_j,t_{m+1/2})}{\partial t}=\frac{u(x_i,y_j, t_{m+1})-u(x_i,y_j,t_{m})}{\Delta t}+\mathcal{O}(\Delta t^2).
\end{equation}
For the nonlinear term $f(x_i,y_j,t_{m+1/2},u(x_i,y_j,t_{m+1/2}))$, we will apply the second-order explicit time discretization method \cite{huang2022preconditioners}
%treat it explicitly within the time discretization as
\begin{eqnarray}\label{lineariedf}\footnotesize
&&f(x_i,y_j,t_{m+1/2},u(x_i,y_j,t_{m+1/2}))\nonumber\\
&=&\frac{3}{2}f(x_i,y_j,t_{m+1/2},u(x_i,t_m))-\frac{1}{2}f(x_i,t_{m+1/2},u(x_i,y_j,t_{m-1}))+\mathcal{O}(\Delta t^2).
\end{eqnarray}
Combining (\ref{4th-discretization})-(\ref{lineariedf}), 
%for $1\leq i\leq N_x$, $1\leq j\leq N_y$, $0\leq m\leq M$,
%and $f_{i,j}^m=f(u(x_i,y_j,t_{m}))$.
the fourth-order discretization coupling the Crank-Nicolson (CN) explicit linearized method for the 2D RSFNRDEs is given by
{\small
\begin{eqnarray}\label{2D-1}
&&u(x_i,y_j,t_{m+1})+\eta_\alpha\sum_{k=-N_x+i}^{i-1}s_{k}^{(\alpha)}u({x_{i-k},y_j,t_{m+1}})
+\eta_\beta\sum_{k=-N_y+j}^{j-1}s_{k}^{(\beta)}u({x_i,y_{j-k},t_{m+1}})\nonumber\\
&=&u(x_i,y_j,t_m)-\eta_\alpha\sum_{k=-N_x+i}^{i-1}s_{k}^{(\alpha)}u(x_{i-k},y_j,t_m)
-\eta_\beta\sum_{k=-N_y+j}^{j-1}s_{k}^{(\beta)}u(x_i,y_{j-k},t_{m})\nonumber\\
&&+\frac{3}{2}\Delta t f(x_i,y_j,t_{m+1/2},u(x_i,y_j,t_{m}))-\frac{1}{2} \Delta t f(x_i,y_j, t_{m+1/2},u(x_i,y_j,t_{m-1}))+{\Delta t}{\rho}_{i,j}^{m},
\end{eqnarray}
}
%\begin{eqnarray}\label{2D-1}
%&&u_{i,j}^{m+1}+\eta_\alpha\sum_{k=-N_x+i}^{i-1}s_{k}^{(\alpha)}u_{i-k,j}^{m+1}
%+\eta_\beta\sum_{k=-N_y+j}^{j-1}s_{k}^{(\beta)}u_{i,j-k}^{m+1}\nonumber\\
%&=&u_{i,j}^{m}-\eta_\alpha\sum_{k=-N_x+i}^{i-1}s_{k}^{(\alpha)}u_{i-k,j}^{m}
%-\eta_\beta\sum_{k=-N_y+j}^{j-1}s_{k}^{(\beta)}u_{i,j-k}^{m}
%+\frac{3}{2}\Delta t f(x_i,y_j,t_{m+1/2},u_{i,j}^{m})\nonumber\\
%&&-\frac{1}{2} \Delta t f(x_i,y_j, t_{m+1/2},u_{i,j}^{m-1})+{\Delta t}{ \rho}_{i,j}^{m},
%\end{eqnarray}
where $1\leq i\leq N_x$, $1\leq j\leq N_y$, $0\leq m\leq M$,
%$u_{i,j}^m=u(x_i,y_j,t_m)$, 
$\eta_\alpha=\frac{K_{\alpha}\Delta t}{2h_x^{\alpha}}$, $\eta_\beta=\frac{K_{\beta}\Delta t}{2h_y^{\beta}}$ and ${\rho}_{i,j}^{m}=\mathcal{O}(\Delta t^2+h_x^4+h_y^4)$. By neglecting the term ${\rho}_{i,j}^{m}$ and replacing $u(x_i,y_j,t_m)$
%$u_{i,j}^{m}$ 
with its numerical approximation $U_{i,j}^{m}$, we obtain the 2D linearized CN fourth-order scheme
\begin{eqnarray}\label{2D-1-2}
&&U_{i,j}^{m+1}+\eta_\alpha\sum_{k=-N_x+i}^{i-1}s_{k}^{(\alpha)}U_{i-k,j}^{m+1}
+\eta_\beta\sum_{k=-N_y+j}^{j-1}s_{k}^{(\beta)}U_{i,j-k}^{m+1}\nonumber\\
&=&U_{i,j}^{m}-\eta_\alpha\sum_{k=-N_x+i}^{i-1}s_{k}^{(\alpha)}U_{i-k,j}^{m}
-\eta_\beta\sum_{k=-N_y+j}^{j-1}s_{k}^{(\beta)}U_{i,j-k}^{m}\nonumber\\
&&+\frac{3}{2}\Delta t f(x_i,y_j,t_{m+1/2},U_{i,j}^{m})
-\frac{1}{2}\Delta t f(x_i,y_j,t_{m+1/2},U_{i,j}^{m-1}).
\end{eqnarray}
%where %$\eta_\alpha$ is defined in Section 3, and
%$\eta_\alpha=\frac{K_{\alpha}\Delta t}{2h_x^{\alpha}}$ and $\eta_\beta=\frac{K_{\beta}\Delta t}{2h_y^{\beta}}$.
Note that scheme (\ref{2D-1-2}) is a two-step method, which requires two starting values. In particular, we set $U_{i,j}^{-1}=U_{i,j}^0$, then scheme (\ref{2D-1-2}) can be implemented.

Now, let $N=N_xN_y$, and define $U^{m}$ and $F^{m}$ as 
%the following two $N$-dimensional vectors
\begin{equation}\nonumber%\label{num-sol}
U^{m}=\big(U_{1,1}^{m},\ldots,U_{N_x,1}^{m},U_{1,2}^{m},\ldots,U_{N_x,2}^{m},\ldots,
U_{1,N_y}^{m},\ldots,U_{N_x,N_y}^{m}\big)^{T}
\end{equation}
and
\begin{eqnarray}\label{source-term}
F^{m}=\Big(f(x_1,y_1,t_{m+1/2},U_{1,1}^m),\ldots, f(x_{N_x},y_1,t_{m+1/2},U_{N_x,1}^m),f(x_1,y_2,t_{m+1/2},U_{1,2}^m),\ldots,\nonumber\\
f(x_{N_x},y_2,t_{m+1/2}, U_{N_x,2}^m),\ldots,f(x_{1},y_{N_y},t_{m+1/2},U_{1,N_y}^m),\ldots,
f(x_{N_x},y_{N_y},t_{m+1/2}, U_{{N_x},{N_y}}^m)\Big)^T.\nonumber
\end{eqnarray}
%\begin{equation}\label{source-term}
%F^{m}=\Big[f_{1,1}^{m},\ldots,f_{N_x,1}^{m},f_{1,2}^{m},\ldots,f_{N_x,2}^{m},\ldots,
%f_{1,N_y}^{m},\ldots,f_{N_x,N_y}^{m}\Big]^{T},
%\end{equation}
%with $f_{i,j}^{m}=f(x_i,y_j,t_{m+1/2},U_{i,j}^m)$.
Then the discrete linearized scheme (\ref{2D-1-2}) can be expressed as
\begin{equation}\label{2D-matrix-form}
(I_{N}+J_{\alpha,\beta})U^{m+1}=(I_{N}-J_{\alpha,\beta})U^{m}+\Delta t(\frac{3}{2}F^{m}-\frac{1}{2}F^{m-1}),\quad 0\leq m\leq M-1,
\end{equation}
where $I_N$ is the identity matrix of order $N$, and
\begin{equation}\nonumber
J_{\alpha,\beta}=\eta_\alpha (I_{N_y}\otimes{A}_{\alpha})+\eta_\beta (A_{\beta}\otimes I_{N_x}),
\end{equation}
in which $\otimes$ denotes the Kronecker product.
%and $A_{\alpha}$ is a symmetric and positive definite (SPD) Toeplitz matrix based on Proposition \ref{coef_prop}, see \cite{ding2023high},
The matrix $A_{\alpha}$ is of the form
\begin{eqnarray}\label{Toeplitz-Walpha}
A_{\alpha}=\left[\begin{array}{ccccc}
s_0^{(\alpha)} & s_{-1}^{(\alpha)} & \cdots & s_{2-N_x}^{(\alpha)} & s_{1-N_x}^{(\alpha)}\\
s_{1}^{(\alpha)} & s_0^{(\alpha)} & s_{-1}^{(\alpha)} & \ddots & s_{2-N_x}^{(\alpha)}\\
\vdots & s_{1}^{(\alpha)} & s_0^{(\alpha)} & \ddots & \vdots\\
s_{N_x-2}^{(\alpha)} & \ddots & \ddots & \ddots & s_{-1}^{(\alpha)}\\
s_{N_x-1}^{(\alpha)} & s_{N_x-2}^{(\alpha)} & \cdots & s_{1}^{(\alpha)} & s_0^{(\alpha)}
\end{array}
\right]
\end{eqnarray}
with $s_{k}^{(\alpha)}$ being defined in (\ref{coesk}). $A_{\beta}$ is defined similarly as $A_{\alpha}$ just with $\alpha$ (resp. $N_x$) replaced by $\beta$ (resp. $N_y$).
%The following lemma is critical for establishing the stability and convergence of the proposed scheme (\ref{2D-matrix-form}).
%\begin{lemma}\label{Grownall}
%(Grownall's inequality \cite{cheng2019novel}) Assume that $\{k_n|n\geq 0\}$ and $\{p_n|n\geq 0\}$ are two nonnegative sequences, and the sequence $\{\phi_n|n\geq 0\}$ satisfies
%\begin{equation}\nonumber
%\phi_0\leq \xi_0,\quad\phi_{n}\leq \xi_0+\sum_{l=0}^{n-1}p_{l}+\sum_{l=0}^{n-1}k_{l}\phi_l,\quad n\geq 1,
%\end{equation}
%where $\xi_0\geq 0$.
%%and $c$ is a positive number.
%Then it holds that
%\begin{equation}\nonumber
%\phi_n\leq\exp\Big(\sum_{l=0}^{n-1}k_{l}\Big)\Big(\xi_0+\sum_{l=0}^{n-1}p_{l}\Big),\quad n\geq 1.
%\end{equation}
%\end{lemma}

%\section{Stability and Convergence}
Next, we will consider the stability and the convergence of the scheme (\ref{2D-matrix-form}) for 2D RSFNRDEs. Firstly, we discuss its unconditional stability. Let ${\widetilde U}_{i,j}^m$ and $U_{i,j}^m$ be two solutions of the scheme (\ref{2D-matrix-form}), and define
\begin{equation}\nonumber
\epsilon_{i,j}^m={\widetilde U}_{i,j}^m-U_{i,j}^m,\quad m=0,1,\ldots,M,i=1,2,\ldots,N_x,j=1,2,\ldots,N_y.
\end{equation}
According to (\ref{2D-1-2}), $\epsilon_{i,j}^m$ satisfies the following equation
\begin{equation}\label{2D-2}
\begin{aligned}
%\begin{eqnarray}\label{2D-2}
&\epsilon_{i,j}^{m+1}+\eta_\alpha\sum_{k=-N_x+i}^{i-1}s_{k}^{(\alpha)}\epsilon_{i-k,j}^{m+1}
+\eta_\beta\sum_{k=-N_y+j}^{j-1}s_{k}^{(\beta)}\epsilon_{i,j-k}^{m+1}\\
=&\epsilon_{i,j}^{m}-\eta_\alpha\sum_{k=-N_x+i}^{i-1}s_{k}^{(\alpha)}\epsilon_{i-k,j}^{m}
-\eta_\beta\sum_{k=-N_y+j}^{j-1}s_{k}^{(\beta)}\epsilon_{i,j-k}^{m}+\frac{3}{2}\Delta t(f(x_i,y_j,t_{m+1/2},{\widetilde U}_{i,j}^m)\\
&-f(x_i,y_j,t_{m+1/2},U_{i,j}^m))-\frac{1}{2}\Delta t(f(x_i,y_j,t_{m+1/2},{\widetilde U}_{i,j}^{m-1})-f(x_i,y_j,t_{m+1/2},U_{i,j}^{m-1})).
%\end{eqnarray}
\end{aligned}
\end{equation}
%{\footnotesize
%\begin{eqnarray}\label{2D-2}
%&&\epsilon_{i,j}^{m+1}+\eta_\alpha\sum_{k=-N_x+i}^{i-1}s_{k}^{(\alpha)}\epsilon_{i-k,j}^{m+1}
%+\eta_\beta\sum_{k=-N_y+j}^{j-1}s_{k}^{(\beta)}\epsilon_{i,j-k}^{m+1}\nonumber\\
%&=&\epsilon_{i,j}^{m}-\eta_\alpha\sum_{k=-N_x+i}^{i-1}s_{k}^{(\alpha)}\epsilon_{i-k,j}^{m}
%-\eta_\beta\sum_{k=-N_y+j}^{j-1}s_{k}^{(\beta)}\epsilon_{i,j-k}^{m}+\frac{3}{2}\Delta t(f(x_i,y_j,t_{m+1/2},{\widetilde U}_{i,j}^m)\\
%&&-f(x_i,y_j,t_{m+1/2},U_{i,j}^m))-\frac{1}{2}\Delta t(f(x_i,y_j,t_{m+1/2},{\widetilde U}_{i,j}^{m-1})-f(x_i,y_j,t_{m+1/2},U_{i,j}^{m-1}))\nonumber.
%\end{eqnarray}
%}
Furthermore, denote
$$
E^{m}=(\epsilon_{1,1}^m,\ldots,\epsilon_{N_x,1}^m,\epsilon_{1,2}^m,\ldots,
\epsilon_{N_x,2}^m,\ldots,\epsilon_{1,N_y}^m,\ldots,\epsilon_{N_x,N_y}^m)^{T},
$$
for $m=0,1,\ldots,M-1$, we have, 
\begin{equation}\label{2D-U-UEQ-matrix}
(I_{N}+J_{\alpha,\beta})E^{m+1}=(I_{N}-J_{\alpha,\beta})E^{m}+\frac{3}{2}\Delta t({\widetilde F}^{m}-F^{m})-\frac{1}{2}\Delta t({\widetilde F}^{m-1}-F^{m-1}),
%\, 0\leq m\leq M-1,
\end{equation}
%\begin{equation}\label{U-UEQ-matrix}
%(I+\eta_\alpha A_{\alpha})E^{m+1}=(I-\eta_\alpha A_{\alpha})E^{m}+\frac{3}{2}\Delta t({\widetilde F}^{m}-F^{m})-\frac{1}{2}\Delta t({\widetilde F}^{m-1}-F^{m-1}),
%\end{equation}
where
\begin{eqnarray}
{\widetilde F}^{m}=\Big(f(x_1,y_1,t_{m+1/2},{\widetilde U}_{1,1}^m),\ldots, f(x_{N_x},y_1,t_{m+1/2},{\widetilde U}_{N_x,1}^m),f(x_1,y_2,t_{m+1/2},{\widetilde U}_{1,2}^m),\ldots,\nonumber\\
f(x_{N_x},y_2,t_{m+1/2},{\widetilde U}_{N_x,2}^m),\ldots,f(x_{1},y_{N_y},t_{m+1/2},{\widetilde U}_{1,N_y}^m),\ldots,f(x_{N_x},y_{N_y},t_{m+1/2},{\widetilde U}_{{N_x},{N_y}}^m)\Big)^T.\nonumber
\end{eqnarray}
Now, we prove the unconditional stability in the sense that
\begin{equation}\nonumber
||E^{m+1}||\leq\exp{(3C_{L}T)}||E^{0}||,\quad m=0,1,\ldots,M-1,
\end{equation}
where $C_L$ is a positive constant and $||E^{m+1}||$ denotes the discrete $L_2$-norm of $E^{m+1}$, i.e.,
$$
||E^{m+1}||=\sqrt{h_xh_y\sum_{i=1}^{N_x}\sum_{j=1}^{N_y}({\epsilon_{i,j}^{m+1}})^2}.
$$
%Obviously, $||E^{m+1}||=\sqrt{h}||E^{m+1}||_2$ with $||\cdot||$ being the $2$-norm and $h$ being the size of the grid.
%and $C_L$ is a positive number.
\begin{theorem}\label{unconditional-stability}
The scheme (\ref{2D-matrix-form}) is unconditionally stable.
%\begin{equation}\nonumber
%||E^{m+1}||\leq\exp{(3C_{L}T)}||E^{0}||,\quad m=0,1,\ldots,M-1.
%\end{equation}
\end{theorem}
\begin{proof}
We first rewrite (\ref{2D-U-UEQ-matrix}) as
\begin{equation}\label{2D-U-UEQ-matrix-2}
E^{m+1}-E^{m}+J_{\alpha,\beta}(E^{m-1}+E^{m})=\frac{3}{2}\Delta t({\widetilde F}^{m}-F^{m})-\frac{1}{2}\Delta t({\widetilde F}^{m-1}-F^{m-1}),
\end{equation}
then left-multiplying $h_xh_y(E^{m+1}+E^{m})^{T}$ on both sides of (\ref{2D-U-UEQ-matrix-2}), we obtain
\begin{eqnarray}\label{U-UEQ-matrix-3}
&&||E^{m+1}||^2-||E^{m}||^{2}+h_xh_y(E^{m+1}+E^{m})^{T} J_{\alpha,\beta}(E^{m-1}+E^{m})\nonumber\\
&=&h_xh_y(E^{m+1}+E^{m})^{T}\left[\frac{3}{2}\Delta t({\widetilde F}^{m}-F^{m})-\frac{1}{2}\Delta t({\widetilde F}^{m-1}-F^{m-1})\right].
\end{eqnarray}
By Proposition \ref{coef_prop}, we can obtain that $A_\alpha$ and $A_\beta$ are SPD through the similar technique given in \cite{huang2023tau}, which combined with $\eta_\gamma>0$ implies that $J_{\alpha,\beta}=\eta_\alpha (I_{N_y}\otimes{A}_{\alpha})+\eta_\beta (A_{\beta}\otimes I_{N_x})$ is also SPD, i.e.,
\begin{equation}\label{A-alpha-SPD}
(E^{m+1}+E^{m})^{T} J_{\alpha,\beta}(E^{m-1}+E^{m})>0.
\end{equation}
Combining (\ref{U-UEQ-matrix-3}) and (\ref{A-alpha-SPD}), we have
\begin{eqnarray}%\label{U-UEQ-matrix-4}
||E^{m+1}||^2-||E^{m}||^{2}&\leq&h_xh_y(E^{m+1}+E^{m})^{T}\left[\frac{3}{2}\Delta t({\widetilde F}^{m}-F^{m})-\frac{1}{2}\Delta t({\widetilde F}^{m-1}-F^{m-1})\right]\nonumber\\
&\leq&\Delta t||E^{m+1}+E^{m}||\cdot||\frac{3}{2}({\widetilde F}^{m}-F^{m})-\frac{1}{2}({\widetilde F}^{m-1}-F^{m-1})||\nonumber\\
&\leq&\Delta t\left(||E^{m+1}||+||E^{m}||\right)\cdot\left[\frac{3}{2}||{\widetilde F}^{m}-F^{m}||+\frac{1}{2}||{\widetilde F}^{m-1}-F^{m-1}||\right]\nonumber.
\end{eqnarray}
Hence,
\begin{eqnarray}\label{U-UEQ-matrix-5}
||E^{m+1}||-||E^{m}||
&\leq&\Delta t\left[\frac{3}{2}||{\widetilde F}^{m}-F^{m}||+\frac{1}{2}||{\widetilde F}^{m-1}-F^{m-1}||\right].
\end{eqnarray}
According to the Lipschitz condition (\ref{Lipschitz-condition}), we have
\begin{equation}\label{lip-1}
\left|f(x_i,y_j,t_{m+1/2},{\widetilde U}_{i,j}^m)-f(x_i,y_j,t_{m+1/2},U_{i,j}^m)\right|\leq L_1|{\widetilde U_{i,j}}^{m}-U_{i,j}^{m}|=L_1|{\epsilon_{i,j}^{m}}|,
\end{equation}
and
%\begin{equation}\label{lip-2}
\begin{equation}
\begin{aligned}\label{lip-2}
\left|f(x_i,y_j,t_{m+1/2},{\widetilde U}_{i,j}^{m-1})-f(x_i,y_j,t_{m+1/2},U_{i,j}^{m-1})\right|\leq L^{'}_1|{\widetilde U}_{i,j}^{m-1}-U_{i,j}^{m-1}|=L^{'}_1|{\epsilon_{i,j}^{m-1}}|.
\end{aligned}
\end{equation}
%\end{equation}
%\begin{equation}\label{lip-1}
%|{\widetilde F}^{m}-F^{m}|\leq L_1|{\widetilde U}^{m}-U^{m}|=L_1|{E^{m}}|,
%\end{equation}
%and
%\begin{equation}\label{lip-2}
%|{\widetilde F}^{m-1}-F^{m-1}|\leq L^{'}_1|{\widetilde U}^{m-1}-U^{m-1}|=L^{'}_1|{E^{m-1}}|.
%\end{equation}
On the basis of (\ref{lip-1}) and (\ref{lip-2}) coupling the definition of the $L_2$-norm of a vector, we have,
\begin{eqnarray}\label{nonlinearF-1}
||{\widetilde F}^{m}-F^{m}||&=&\sqrt{h_xh_y\sum_{i=1}^{N_x}\sum_{j=1}^{N_y} \left|f(x_i,y_j,t_{m+1/2},{\widetilde U}_{i,j}^m)-f(x_i,y_j,t_{m+1/2},U_{i,j}^m)\right|^2}\nonumber\\
&\leq&\sqrt{L_1 h_xh_y\sum_{i=1}^{N_x}\sum_{j=1}^{N_y}|{\widetilde U}_{i,j}^m-U_{i,j}^m|^2}\nonumber\\
&=&\sqrt{L_1 h_xh_y\sum_{i=1}^{N_x}\sum_{j=1}^{N_y}(\epsilon_{i,j}^m)^2}\nonumber\\
&=&\sqrt{L_1}||E^{m}||.
\end{eqnarray}
Similarly, we have
\begin{eqnarray}\label{nonlinearF-2}
||{\widetilde F}^{m-1}-F^{m-1}||\leq\sqrt{L^{'}_1}||E^{m-1}||.
\end{eqnarray}
Combining with (\ref{U-UEQ-matrix-5}), (\ref{nonlinearF-1}) and (\ref{nonlinearF-2}), we arrive
\begin{eqnarray}\nonumber
||E^{m+1}||-||E^{m}||&\leq&\Delta t \left(\frac{3\sqrt{L_1}}{2}||E^{m}||+\frac{\sqrt{L^{'}_1}}{2}||E^{m-1}||\right),
\end{eqnarray}
namely,
\begin{eqnarray}\label{Em+1toEm}
||E^{m+1}||-||E^{m}||&\leq&\Delta t C_{L} \left(||E^{m}||+||E^{m-1}||\right),
\end{eqnarray}
where $C_L=\max\{\frac{3\sqrt{L_1}}{2},\frac{\sqrt{L^{'}_1}}{2}\}$. Summing up (\ref{Em+1toEm}) for $m$ from $1$ to $k$, we have
\begin{eqnarray}\label{E-k+1toE0}
||E^{k+1}||&\leq&||E^{1}||+\Delta t C_{L} \left(\sum_{p=1}^{k}||E^{p}||+\sum_{p=1}^{k}||E^{p-1}||\right)\nonumber\\
&\leq&||E^{0}||+3\Delta t C_{L}\sum_{p=1}^{k}||E^{p}||.
\end{eqnarray}
Applying Gronwall's inequality (see \cite{cheng2019novel}) to (\ref{E-k+1toE0}), we obtain
\begin{eqnarray}\nonumber%\label{Gronwall-inequality-result}
||E^{m+1}||&\leq&\exp(3C_{L}k\Delta t)||E^{0}||\leq\exp(3C_{L}T)||E^{0}||,
\end{eqnarray}
which ends the proof.
\end{proof}
%\end{theorem}

Next, we will prove the convergence of the scheme (\ref{2D-matrix-form}).
\begin{theorem}
Let $U_{i,j}^m$ be the solution of the scheme (\ref{2D-matrix-form}), and $u(x_i,y_j,t_m)$
%$u_{i,j}^m$ 
be the exact solution of (\ref{2D-RSFNRDEs-eq}), then there exists two positive constants $C^{*}$ and $C_{R}$
such that
\begin{equation}\nonumber
||u^{m+1}-U^{m+1}||\leq TC^{*}\exp(3TC_{R})(\Delta t^2+h_x^4+h_y^4),\quad m=0,1,\ldots,M-1.
\end{equation}
%where $||\cdot||$ stands for the discrete $L_2$-norm.
\end{theorem}
\begin{proof}
Denote 
\begin{equation}\label{psi_error}
\psi_{i,j}^m=u(x_i,y_j,t_m)-U_{i,j}^m,
\end{equation}
%Let $\psi_{i,j}^m=u_{i,j}^m-U_{i,j}^m$, 
from (\ref{2D-1}) and (\ref{2D-1-2}), $\psi_{i,j}^m$ satisfies the perturbation equation
\begin{eqnarray}\label{perturbation-equation}
&&\psi_{i,j}^{m+1}+\eta_\alpha\sum_{k=-N_x+i}^{i-1}s_{k}^{(\alpha)}\psi_{i-k,j}^{m+1}
+\eta_\beta\sum_{k=-N_y+j}^{j-1}s_{k}^{(\beta)}\psi_{i,j-k}^{m+1}\nonumber\\
&=&\psi_{i,j}^{m}-\eta_\alpha\sum_{k=-N_x+i}^{i-1}s_{k}^{(\alpha)}\psi_{i-k,j}^{m}
-\eta_\beta\sum_{k=-N_y+j}^{j-1}s_{k}^{(\beta)}\psi_{i,j-k}^{m}\nonumber\\
&&+\frac{3}{2}\Delta t\big(f(x_i,y_j,t_{m+1/2},u(x_i,y_j,t_m))-f(x_i,y_j,t_{m+1/2},U_{i,j}^{m})
\big)\nonumber\\
&&-\frac{1}{2}\Delta t\big(f(x_i,y_j,t_{m+1/2},u(x_i,y_j,t_{m-1}))-f(x_i,y_j,t_{m+1/2},U_{i,j}^{m-1})
\big)+{\Delta t}{ \rho}_{i,j}^{m}.
\end{eqnarray}
%\begin{eqnarray}\label{perturbation-equation}
%&&\psi_{i,j}^{m+1}+\eta_\alpha\sum_{k=-N_x+i}^{i-1}s_{k}^{(\alpha)}\psi_{i-k,j}^{m+1}
%+\eta_\beta\sum_{k=-N_y+j}^{j-1}s_{k}^{(\beta)}\psi_{i,j-k}^{m+1}\nonumber\\
%&=&\psi_{i,j}^{m}-\eta_\alpha\sum_{k=-N_x+i}^{i-1}s_{k}^{(\alpha)}\psi_{i-k,j}^{m}
%-\eta_\beta\sum_{k=-N_y+j}^{j-1}s_{k}^{(\beta)}\psi_{i,j-k}^{m}\nonumber\\
%&&+\frac{3}{2}\Delta t\big(f(x_i,y_j,t_{m+1/2},u_{i,j}^{m})-f(x_i,y_j,t_{m+1/2},U_{i,j}^{m})
%\big)\nonumber\\
%&&-\frac{1}{2}\Delta t\big(f(x_i,y_j,t_{m+1/2},u_{i,j}^{m})-f(x_i,y_j,t_{m+1/2},U_{i,j}^{m})
%\big)+{\Delta t}{ \rho}_{i,j}^{m}.
%\end{eqnarray}
Denote
{\footnotesize
\begin{eqnarray}
{\hat F}^{m}=\Big(f(x_1,y_1,t_{m+1/2},u(x_1,y_1,t_m)),\ldots, f(x_{N_x},y_1,t_{m+1/2},u(x_{N_x},y_1,t_m)),f(x_1,y_2,t_{m+1/2},u(x_1,y_2,t_m)),\ldots,\nonumber\\
f(x_{N_x},y_2,t_{m+1/2},u(x_{N_x},y_2,t_m)),\ldots,f(x_{1},y_{N_y},t_{m+1/2},u(x_1,y_{N_y},t_m)),
\ldots,f(x_{N_x},y_{N_y},t_{m+1/2},u(x_{N_x},y_{N_y},t_m))\Big)^T,\nonumber
\end{eqnarray}
}
%Denote
%\begin{eqnarray}
%{\hat F}^{m}=\Big(f(x_1,y_1,t_{m+1/2},u_{1,1}^m),\ldots, f(x_{N_x},y_1,t_{m+1/2},u_{N_x,1}^m),f(x_1,y_2,t_{m+1/2},u_{1,2}^m),\ldots,\nonumber\\
%f(x_{N_x},y_2,t_{m+1/2},u_{N_x,2}^m),\ldots,f(x_{1},y_{N_y},t_{m+1/2},u_{1,N_y}^m),
%\ldots,f(x_{N_x},y_{N_y},t_{m+1/2},u_{{N_x},{N_y}}^m)\Big)^T,\nonumber
%\end{eqnarray}
and
$$
\rho^{m}=\Big(\rho_{1,1}^{m},\ldots,\rho_{N_x,1}^{m},\rho_{1,2}^{m},\ldots,\rho_{N_x,2}^{m},
\ldots,\rho_{1,N_y}^{m},\ldots,\rho_{N_x,N_y}^{m}\Big)^{T},
$$
%$$
%{\hat F}^{m}=(f(x_1,t_{m+1/2},u_1^m), f(x_2,t_{m+1/2},u_2^m), \ldots,f(x_N,t_{m+1/2},u_N^m))^T,
%$$
then (\ref{perturbation-equation}) can be written in matrix form as follows:
\begin{equation}\label{perturbation-equation2}
(I_{N}+J_{\alpha,\beta})\Psi^{m+1}=(I_{N}-J_{\alpha,\beta})\Psi^{m}+\frac{3}{2}\Delta t({\hat F}^{m}-F^{m})-\frac{1}{2}\Delta t({\hat F}^{m-1}-F^{m-1})+\Delta t \rho^{m},
\end{equation}
where
$$
\Psi^{m}=\Big(\psi_{1,1}^{m},\ldots,\psi_{N_x,1}^{m},\psi_{1,2}^{m},\ldots,\psi_{N_x,2}^{m},
\ldots,\psi_{1,N_y}^{m},\ldots,\psi_{N_x,N_y}^{m}\Big)^{T}
$$
with $\Psi^{0}={\bf 0}$.
%and $c$ being a vector.
A direct application of Eqs. (\ref{2D-U-UEQ-matrix-2})-(\ref{U-UEQ-matrix-5}) in Theorem \ref{unconditional-stability} to Eq. (\ref{perturbation-equation2}) yields
%Similar to the proof of Theorem \ref{unconditional-stability}, we get
\begin{eqnarray}\footnotesize\label{global-error1}
||\Psi^{m+1}||-||\Psi^{m}||
\leq\Delta t\left[\frac{3}{2}||{\hat F}^{m}-F^{m}||+\frac{1}{2}||{\hat F}^{m-1}-F^{m-1}||\right]+\Delta t||\rho^{m}||.
\end{eqnarray}
Using the Lipschitz condition (\ref{Lipschitz-condition}), (\ref{global-error1}) can be rewritten as
\begin{eqnarray}\label{global-error2}
||\Psi^{m+1}||-||\Psi^{m}||&\leq&\Delta t \left(\frac{3\sqrt{L_2}}{2}||\Psi^{m}||+\frac{\sqrt{L^{'}_2}}{2}||\Psi^{m-1}||\right)+\Delta t||\rho^{m}||\nonumber\\
&\leq&\Delta t C_R\Big(||\Psi^{m}||+||\Psi^{m-1}||\Big)+\Delta t||\rho^{m}||,
\end{eqnarray}
where $C_{R}=\max\{\frac{3\sqrt{L_2}}{2},\frac{\sqrt{L^{'}_2}}{2}\}$. Summing up (\ref{global-error2}) for $m$ from $1$ to $k$ and noticing $||\Psi^{0}||=0$, we obtain
\begin{eqnarray}\label{global-error3}
||\Psi^{k+1}||&\leq&||\Psi^{1}||+\Delta t C_{R} \Bigg(\sum_{p=1}^{k}||\Psi^{p}||+\sum_{p=1}^{k}||\Psi^{p-1}||\Bigg)+\Delta t\sum_{p=1}^{k}||\rho^{p}||\nonumber\\
&\leq&||\Psi^{0}||+3\Delta t C_{R}\sum_{p=1}^{k}||\Psi^{p}||+\Delta t\sum_{p=1}^{k}||\rho^{p}||\nonumber\\
&=&3\Delta t C_{R}\sum_{p=1}^{k}||\Psi^{p}||+\Delta t\sum_{p=1}^{k}||\rho^{p}||.
\end{eqnarray}
Applying Gronwall's inequality in \cite{cheng2019novel} to (\ref{global-error3}), we have
\begin{eqnarray}\label{Gronwall-global-error}
||\Psi^{m+1}||&\leq&\exp(3 k \Delta t C_{R})\Big(\Delta t\sum_{p=1}^{k}||\rho^{p}||\Big)\nonumber\\
&\leq&TC^{*}\exp(3TC_{R})(\Delta t^2+h_x^4+h_y^4)\nonumber,
\end{eqnarray}
which completes the proof.
\end{proof}

\section{$\tau$-matrix based preconditioner and its spectral analysis}
In this section, we propose a $\tau$-matrix based preconditioner for the system (\ref{2D-matrix-form}), which is constructed based on the generation function $S_\gamma(\theta)$ of $A_\gamma$. Knowing that $e^{{\pm\bf i}\theta}=\cos(\theta)\pm{\bf i}\sin(\theta)$, then we have
\begin{equation}\label{Toeplitz-S}
\begin{aligned}
S_\gamma(\theta)&=\left[1+\frac{\gamma}{24}(2-e^{{\bf i}\theta}-e^{-{\bf i}\theta})\right](2-e^{{\bf i}\theta}-e^{-{\bf i}\theta})^\frac{\gamma}{2},~~ (\gamma=\alpha,\beta)\\
&=\left[1+\frac{\gamma}{24}(2-2\cos\theta)\right](2-2\cos\theta)^\frac{\gamma}{2}.\\
%&=Q_\gamma(\theta)G_\gamma(\theta),
\end{aligned}
\end{equation}
%where $Q_\gamma(\theta)=1+\frac{\gamma}{24}(2-2\cos\theta)$ and $G_\gamma(\theta)=(2-2\cos\theta)^\frac{\gamma}{2}$.

Furthermore, we denote
\begin{equation}\label{Toeplitz-Q}
Q_\gamma=\left(\begin{array}{ccccc}
1 &  & & & \\
 & 1 &  & & \\
&  & \ddots &  & \\
& &  & 1 &  \\
& & &  & 1
\end{array}\right)+\frac{\gamma}{24}\left(\begin{array}{ccccc}
2 & -1 & & & \\
-1 & 2 & -1 & & \\
& \ddots & \ddots & \ddots & \\
& & -1 & 2 & -1 \\
& & & -1 & 2
\end{array}\right) \in \mathbb{R}^{N_z \times N_z}
\end{equation}
with $(\gamma,z)=(\alpha,x)$ or $(\beta,y)$. From the definition of generating function given in \cite{Jin1}, we know that the Toeplitz matrix $Q_\gamma$ is generated by the function $1+\frac{\gamma}{24}(2-2\cos\theta)$.
Again, from (\ref{Riesz_coef}), we see that $(2-2\cos\theta)^\frac{\gamma}{2}$ is the generating function of $\hat A _\gamma$ defined similarly as $A_\gamma$ in (\ref{Toeplitz-Walpha}), with the entries $s_i^{(\gamma)}$ replaced by $g_i^{(\gamma)}$, where
\begin{equation}
g_i^{(\gamma)}=\frac{1}{2\pi}\int_{-\pi}^{\pi}(2-e^{{\bf i}\theta}-e^{{\bf -i}\theta})^{\frac{\gamma}{2}e^{-{\bf i} k\theta}}d\theta=\frac{(-1)^{k}\Gamma(\gamma+1)}{\Gamma(\frac{\gamma}{2}-k+1)\Gamma(\frac{\gamma}{2}+k+1)}.
\end{equation}
More properties about $g_i^{(\gamma)}$ can be found in \cite{CD2}.

Moreover, inspired by the proposal in \cite{bini1990new}, it has been shown that $\tau(\hat A _\gamma)$ is a good preconditioner for the Toeplitz matrix $\hat A _\gamma$ in \cite{zhang2022fast}. Thus, in this paper, we define the following preconditioner for approximating $A_\gamma$:
\begin{equation}\label{ptau0}
P_{\gamma}:=Q_\gamma \tau(\hat A _\gamma).
\end{equation}

We remark that both $Q_\gamma$ and $\tau(\hat A _\gamma)$ are $\tau$ matrices possessing a symmetric structure, which can be diagonalized by the DST matrix $S_{N_z}$ and the elements of $S_{N_z}$ are calculated by
$$
[S_{N_z}]_{j,k}=\sqrt{\frac{2}{N_z+1}}\sin(\frac{\pi jk}{N_z+1}),\quad 1\leq j,k\leq N_z,
$$
see \cite{bini1990new} for more details. Therefore, $Q_\gamma$ and $\tau(\hat A _\gamma)$ commute, which shows that $P_{\gamma}$ is symmetric and can also be diagonalized by the DST matrix, i.e.,
\begin{equation}\label{decomp_pgamma}
P_{\gamma}=S_{N_z} \Lambda_{P_\gamma} S_{N_z},
\end{equation}
where
\begin{equation}\label{eig_pgamma}
\Lambda_{P_\gamma}=\Lambda_{Q_\gamma} \Lambda_{\tau(\hat A _\gamma)}.
\end{equation}
The matrices $\Lambda_{Q_\gamma}$ and $\Lambda_{\tau(\hat A _\gamma)}$ contain the eigenvalues of the matrices $Q_\gamma$ and $\tau(\hat A _\gamma)$, respectively.

Then, our $\tau$-matrix based preconditioner for the coefficient matrix $I_{N}+J_{\alpha,\beta}$ in (\ref{2D-matrix-form}) is defined as
\begin{equation}\label{ptau}
P_{\tau}=I_N+P(J_{\alpha,\beta}),
\end{equation}
where
%\begin{equation}
%{\hat J}_{\alpha,\beta}=\eta_\alpha (I_{N_y}\otimes{\hat A}_{\alpha})+\eta_\beta ({\hat A}_{\beta}\otimes I_{N_x}),
%\end{equation}
%and
\begin{equation}\label{tau_pre1}
P(J_{\alpha,\beta})=\eta_\alpha\big(I_{N_y}\otimes P_{\alpha}\big)
+\eta_\beta\big(P_{\beta}\otimes I_{N_x}\big)
\end{equation}
with $P_{\alpha}$ and $P_{\beta}$ defined in (\ref{ptau0}).

Based on the Eqs. (\ref{decomp_pgamma})-(\ref{tau_pre1}), it is easy to check that $P_{\tau}$ has the following decomposition
\begin{equation}\label{p_decomp}
P_{\tau}=(S_{N_y} \otimes S_{N_x}) \left(I_N + \eta_\alpha(I_{N_y}\otimes \Lambda_{P_\alpha}) + \eta_\beta(\Lambda_{P_\beta} \otimes I_{N_x}) \right) (S_{N_y} \otimes S_{N_x}).
\end{equation}
Therefore, for a given vector $v$,
%for an arbitrary given vector $v$, 
the matrix-vector multiplication $P^{-1}_{\tau}v$ can be implemented within ${\mathcal O}(N_xN_y\log(N_xN_y))$ operations using fast DSTs.

In order to show the proposed preconditioner $P_{\tau}$ is well-defined, the invertibility of $P_{\tau}$ is given in Lemma \ref{pre_SPD} as follows.
%sine transform.
\begin{lemma}\label{pre_SPD}
The $\tau$-matrix based preconditioner $P_{\tau}$ defined in (\ref{ptau}) is SPD and thus $P_{\tau}$ is invertible.
\end{lemma}
\begin{proof}
%By straightforward calculation, it concludes from (\ref{Toeplitz-Q})
It is obvious that $Q_\gamma$ is symmetric from (\ref{Toeplitz-Q}). Then, the matrix $Q_\gamma$ is SPD since it is diagonally dominant with positive diagonal entries. Moreover, from \cite{zhang2022fast}, we also know that $\tau(\hat A _\gamma)$ is SPD. Owing to the commutability of matrices $Q_\gamma$ and $\tau(\hat A _\gamma)$, we conclude that $P_{\gamma}$ is SPD matrix. Finally, combining (\ref{ptau})-(\ref{p_decomp}) with the properties of Kronecker product completes the proof.
\end{proof}

In the following, a few auxiliary lemmas are introduced to study the eigenvalues of the preconditioned matrix $P_\tau^{-1}(I_N+J_{\alpha,\beta})$.
%with respect to the preconditioner $P_{\tau}$.

%In order to discuss the spectrum of the preconditioned matrix with respect to the preconditioner $P_{\tau}$, a few auxiliary lemmas are needed.

%\begin{lemma}\label{inequality1}\cite{LNS2}
%For positive numbers $\xi_i$ and $\epsilon_i$ with $1\leq i\leq N$, it holds that
%\begin{equation}
%\min_{1\leq i\leq N}\frac{\xi_i}{\epsilon_i}\leq\left(\sum_{i=1}^N\epsilon_i\right)^{-1}\left(\sum_{i=1}^N\xi_i\right)\leq \max_{1\leq i\leq N}\frac{\xi_i}{\epsilon_i}.
%\end{equation}
%\end{lemma}
\begin{lemma}{\rm (\cite{zhang2022fast})}\label{pre2FCD}
Assume that $\lambda(\tau^{-1}(\hat A _\gamma)\hat A _\gamma)$ be the eigenvalues of the matrix $\tau^{-1}(\hat A _\gamma)\hat A _\gamma$, then it holds that
\begin{equation}\nonumber
\frac{1}{2}<\lambda(\tau^{-1}(\hat A _\gamma)\hat A _\gamma)<\frac{3}{2}.
\end{equation}
\end{lemma}

\begin{lemma}\label{generatingAB}{\rm (\cite{serra1994preconditioning,serra1998extreme})}
Let $T_1$ and $T_2$ be the Toeplitz matrices generated by $f_{1}(\theta)$ and $f_{2}(\theta)$ with $f_{1}(\theta), f_{2}(\theta)\in L^{1}([-\pi,\pi]^2)$ being real almost everywhere, respectively. Then
\begin{equation}\nonumber
\widetilde{l}={\rm ess}\mathop{\inf}\limits_{\theta}\frac{f_2 (\theta)}{f_1(\theta)}\leq\lambda(T_{1}^{-1}T_2)\leq {\rm ess}\mathop{\sup}\limits_{\theta}\frac{f_2 (\theta)}{f_1(\theta)}= \widetilde{L},
\end{equation}
where {\rm essinf} and {\rm esssup} denote the essential infimum and the essential supremum, respectively. Moreover, if $\widetilde{l}<\widetilde{L}$, then the above estimates are strict, i.e.,
\begin{equation}\nonumber
\widetilde{l}<\lambda(T_{1}^{-1}T_2)<\widetilde{L}.
\end{equation}
\end{lemma}

\begin{lemma}\label{generatingvalueAB}
Let $S_\gamma(\theta)$ and $G_\gamma(\theta)$ be the generating functions of $A_\gamma$ and $\hat{A}_\gamma$, respectively. For $1<\gamma<2$, it holds that
\begin{equation}\nonumber
{\rm ess}\mathop{\inf}\limits_{\theta}\frac{S_\gamma(\theta)}{G_\gamma(\theta)}\geq 1 \quad \mbox{and}\quad{\rm ess}\mathop{\sup}\limits_{\theta}\frac{S_\gamma(\theta)}{G_\gamma(\theta)}\leq \frac{4}{3}.
\end{equation}
\end{lemma}
\begin{proof}
From (\ref{Toeplitz-S}), we have
\begin{eqnarray}\label{Stheta}
S_\gamma(\theta)
&=&\left[1+\frac{\gamma}{24}(2-e^{{\bf i}\theta}-e^{-{\bf i}\theta})\right](2-e^{{\bf i}\theta}-e^{-{\bf i}\theta})^\frac{\gamma}{2}\nonumber\\
&=&\left[1+\frac{\gamma}{24}(2-2\cos\theta)\right](2-2\cos\theta)^\frac{\gamma}{2}\nonumber\\
&=&\left[1+\frac{\gamma}{6}\sin^2(\frac{\theta}{2})\right]\cdot\left[2\sin(\frac{\theta}{2})\right]^{\gamma}.
\end{eqnarray}
Note that
\begin{eqnarray}\label{Gtheta}
G_\gamma(\theta)&=&\sum_{k=-\infty}^{\infty}g_{k}^{(\gamma)}e^{{\bf i}k\theta}
=(2-e^{{\bf i}\theta}-e^{-{\bf i}\theta})^\frac{\gamma}{2}=\left[2\sin(\frac{\theta}{2})\right]^{\gamma},
\end{eqnarray}
then combining (\ref{Stheta}) and (\ref{Gtheta}), we obtain
\begin{equation}\nonumber
\frac{S_\gamma(\theta)}{G_\gamma(\theta)}=1+\frac{\gamma}{6}\sin^2(\frac{\theta}{2}).
\end{equation}
Hence,
\begin{equation}\nonumber
\hbox{ess\,$\mathop{\inf}\limits_{\theta}$}\frac{S_\gamma(\theta)}{G_\gamma(\theta)}\geq 1+\frac{1}{6}\cdot0=1,\quad \mbox{and}\quad \mbox{ess\,$\mathop{\sup}\limits_{\theta}$}\frac{S_\gamma(\theta)}{G_\gamma(\theta)}\leq 1+\frac{2}{6}\cdot1=\frac{4}{3},
\end{equation}
which completes the proof.
\end{proof}

From Lemmas \ref{generatingAB} and \ref{generatingvalueAB}, we conclude that
\begin{equation}\label{ineq_1}
1<\lambda(\hat{A}_\gamma^{-1}A_\gamma)<\frac{4}{3}.
\end{equation}

\begin{lemma}\label{1D_bound}
Let $\lambda(P_\gamma^{-1}A_\gamma)$ be an eigenvalue of the matrix $P_\gamma^{-1}A_\gamma$, then
\begin{equation}\nonumber
\frac{3}{8}<\lambda(P_\gamma^{-1}A_\gamma)<2.
\end{equation}
\end{lemma}
\begin{proof}
Since $P_\gamma^{-1}A_\gamma$ is similar to the symmetric matrix $P_\gamma^{-\frac{1}{2}}A_\gamma P _\gamma^{-\frac{1}{2}}$, it is equivalent to consider the eigenvalues of $P_\gamma^{-\frac{1}{2}}A_\gamma P _\gamma^{-\frac{1}{2}}$. Then, for arbitrary non-zero vector $x\in {\mathbb R}^N$,
\begin{equation}\label{lambdaquadratic}
\frac{x^T P_\gamma^{-\frac{1}{2}}A_\gamma P _\gamma^{-\frac{1}{2}} x}{x^T x} \xlongequal[]{y=P_\gamma^{-\frac{1}{2}}x}\frac{y^T A_\gamma y}{y^T P_\gamma y}=\frac{y^T A_\gamma y}{y^T \hat{A}_\gamma y}\cdot \frac{y^T  \hat{A}_\gamma y}{y^T \tau(\hat{A}_\gamma) y}\cdot \frac{y^T \tau(\hat{A}_\gamma) y}{y^T P_\gamma y}.
\end{equation}
From (\ref{ineq_1}) and Rayleigh quotient theorem, see \cite{horn2012matrix}, we know that
\begin{equation}\label{ineq_2}
1<\lambda_{\min}(\hat{A}_\gamma^{-1}A_\gamma)\leq\frac{y^T A_\gamma y}{y^T \hat{A}_\gamma y}\leq\lambda_{\max}(\hat{A}_\gamma^{-1}A_\gamma)<\frac{4}{3}.
\end{equation}
On the other hand, Lemma \ref{pre2FCD} implies
\begin{equation}\label{ineq_4}
\frac{1}{2}<\frac{y^T  \hat{A}_\gamma y}{y^T \tau(\hat{A}_\gamma) y}<\frac{3}{2}.
\end{equation}
By straightforward calculation, it concludes from (\ref{Toeplitz-Q}) that $\lambda({Q_{\gamma}})\in(1,\frac{4}{3})$. Moreover, from (\ref{ptau0}) and the fact that $Q_\gamma$ and $\tau(\hat A _\gamma)$ commute, we have
\begin{equation}\label{ineq_3}
\begin{aligned}
\frac{3}{4}<\lambda_{\min}(Q_\gamma^{-1})=\lambda_{\min}(P_\gamma^{-1}\tau(\hat{A}_\gamma))\leq\frac{y^T \tau(\hat{A}_\gamma) y}{y^T P_\gamma z}\leq\lambda_{\max}(P_\gamma^{-1}\tau(\hat{A}_\gamma))
=\lambda_{\max}(Q_\gamma^{-1})<1.\\
\end{aligned}
\end{equation}
At last, combining (\ref{lambdaquadratic})-(\ref{ineq_3}), we deduce that
\begin{equation}\nonumber
\frac{3}{8}=1\cdot\frac{1}{2}\cdot\frac{3}{4}<\lambda(P_\gamma^{-1}A_\gamma)<\frac{4}{3}\cdot\frac{3}{2}\cdot1  =2,
\end{equation}
which completes the proof.
\end{proof}

Moreover, since
\begin{eqnarray}\nonumber
(I_{N_y}\otimes P_\alpha\big)^{-1}(I_{N_y}\otimes A_\alpha)
=(I_{N_y}^{-1}\otimes P_\alpha^{-1})(I_{N_y}\otimes A_\alpha)
=I_{N_y}\otimes P_\alpha^{-1}A_\alpha,
\end{eqnarray}
we have
\begin{equation}\label{eig-AinvA-1}
\frac{3}{8}<\lambda\big(I_{N_y}\otimes P_\alpha^{-1}A_\alpha\big)<2.
\end{equation}
Similarly,
\begin{equation}\label{eig-AinvA-2}
\frac{3}{8}<\lambda\big(P_\beta^{-1}A_\beta\otimes I_{N_x}\big)<2.
\end{equation}
By Lemma \ref{1D_bound}, one can show that the spectrum of the preconditioned matrix $P_\tau^{-1}(I_N+J_{\alpha,\beta})$ is uniformly bounded in the interval $(3/8,2)$ for all $N$.
%the uniform bound of the eigenvalues of the preconditioned matrix $P_\tau^{-1}(I_N+J_{\alpha,\beta})$ can be proved in the following theorem.
\begin{theorem}\label{eigenprematrix}
Let $\lambda(P_\tau^{-1}(I_N+J_{\alpha,\beta}))$ be an eigenvalue of the matrix $P_\tau^{-1}(I_N+J_{\alpha,\beta})$, then
\begin{equation}\nonumber
\frac{3}{8}<\lambda(P_\tau^{-1}(I_N+J_{\alpha,\beta}))<2.
\end{equation}
\end{theorem}
\begin{proof}
By the Rayleigh quotient theorem and (\ref{eig-AinvA-1})-(\ref{eig-AinvA-2}), it holds that
\begin{equation}\nonumber%\label{tensor_ineq_1}
\frac{3}{8}<\lambda_{\min}\big(I_{N_y}\otimes P_{\alpha}^{-1}A_\alpha\big)\leq\frac{v^{T}(I_{N_y}\otimes{A}_{\alpha})v}{v^{T} (I_{N_y}\otimes P_{\alpha}) v}\leq \lambda_{\max}\big(I_{N_y}\otimes P_{\alpha}^{-1}A_\alpha\big)<2,
\end{equation}
and
\begin{equation}\nonumber%\label{tensor_ineq_2}
\frac{3}{8}<\lambda_{\min}\big(P_{\beta}^{-1}A_\beta\otimes I_{N_x}\big)\leq\frac{v^{T}( A_{\beta}\otimes I_{N_x})v}{v^{T} \big(P_{\beta}\otimes I_{N_x}\big)v}\leq \lambda_{\max}\big(P_{\beta}^{-1}A_\beta\otimes I_{N_x}\big)<2.
\end{equation}
Then
\begin{equation}\label{tensor_ineq_1-1}
\frac{3}{8}v^{T}\eta_\alpha(I_{N_y}\otimes P_{\alpha}) v< v^{T}\eta_\alpha(I_{N_y}\otimes{A}_{\alpha})v<2 v^{T}\eta_\alpha(I_{N_y}\otimes P_{\alpha}) v,
\end{equation}
and
\begin{equation}\label{tensor_ineq_2-2}
\frac{3}{8} v^{T}\eta_\beta\big(P_{\beta}\otimes I_{N_x}\big)v< v^{T}\eta_\beta(A_{\beta}\otimes I_{N_x})v<2v^{T}\eta_\beta\big(P_{\beta}\otimes I_{N_x}\big)v.
\end{equation}
Furthermore, recall that
$$
J_{\alpha,\beta}=\eta_\alpha (I_{N_y}\otimes{A}_{\alpha})+\eta_\beta (A_{\beta}\otimes I_{N_x})
$$
and
$$
P_{\tau}=I_{N}+\eta_\alpha\big(I_{N_y}\otimes P_{\alpha}\big)
+\eta_\beta\big(P_{\beta}\otimes I_{N_x}\big),
$$
which combined with (\ref{tensor_ineq_1-1}) and (\ref{tensor_ineq_2-2}) imply
\begin{equation}\nonumber
\frac{3}{8}v^{T}P_{\tau}v<v^{T}(I_{N}+J_{\alpha,\beta})v<2v^{T}P_{\tau}v.
\end{equation}
Thus,
\begin{eqnarray}\nonumber
\frac{3}{8}<\frac{v^{T}(I_N+J_{\alpha,\beta})v}{v^{T}P_{\tau}v}<2,
\end{eqnarray}
namely,
\begin{equation}\nonumber
\lambda_{\min}\Big(P_{\tau}^{-1}(I_{N}+J_{\alpha,\beta})\Big)=\min\frac{v^{T}(I_N+J_{\alpha,\beta})v}{v^{T}P_{\tau} v}>\frac{3}{8},
\end{equation}
and
\begin{equation}\nonumber
\lambda_{\max}\Big(P_{\tau}^{-1}(I_{N}+J_{\alpha,\beta})\Big)=\max\frac{v^{T}(I_N+J_{\alpha,\beta})v}{v^{T}P_{\tau} v}<2,
\end{equation}
which completes the proof.
\end{proof}

Since the spectra of the preconditioned matrix are uniformly bounded, the PCG method for solving (\ref{2D-matrix-form}) converges linearly and the number of iterations is independent of the mesh sizes. Meanwhile, it is worth emphasizing that the construction of $P_{\tau}$ defined in (\ref{ptau}) and the spectral analysis of the corresponding preconditioned matrix are different from that proposed in \cite{huang2023tau}. The main differences consist of the following aspects:
\begin{itemize}
\item[(i)] Constructing $P_{\tau}$ is based on the generating functions of the Toeplitz matrices $A_{\alpha}$ and $A_{\beta}$, while the $\tau$-preconditioner proposed in \cite{huang2023tau} is directly designed from the entries of the symmetric Toeplitz matrices $A_{\alpha}$ and $A_{\beta}$.
\item[(ii)] With the help of the generating functions of $A_{\alpha}$ and $A_{\beta}$, all eigenvalues of the preconditioned matrix $P_{\tau}^{-1}(I_{N}+J_{\alpha,\beta})$ are proved to be strictly greater than $3/8$ and less than $2$, as opposed to the technique in \cite{huang2023tau}, where the related analysis relies on the properties of the matrix entries $s_k^{(\alpha)}$ and $s_k^{(\beta)}$.
    \item[(iii)] As shown in \cite{huang2021spectral}, the monotonicity of the coefficients of low-order schemes plays an important role in proving the uniform boundedness of the spectrum of the preconditioned matrix. However, the coefficients of high-order schemes usually lack this critical property, so for the $\tau$-preconditioner proposed in \cite{huang2023tau} for high-order schemes, the uniform boundedness of the spectrum of the preconditioned matrix is valid only for $1<\alpha, \beta<1.4674$ by using the technique given in \cite{huang2021spectral}. As a remedy, we propose a new $\tau$-matrix based preconditioner $P_{\tau}$ which guarantees all eigenvalues of the preconditioned matrix are uniformly bounded for all $1<\alpha, \beta<2$ (see Theorem 4.1 for more details).
\end{itemize}

%We end this section with the following remark.
%\begin{remark}
%Theorem \ref{eigenprematrix} implies that the spectra of the preconditioned matrix are uniformly bounded, and the smallest eigenvalue of the preconditioned matrix is bounded away from $0$. Therefore, for solving the linear system (\ref{2D-matrix-form}), the PCG method converges linearly, which guarantees that the number of iterations is independent of the size of the coefficient matrix.
%\end{remark}
\section{Numerical experiments}
In this section, some numerical results are reported to verify the accuracy of the scheme (\ref{2D-matrix-form}) and the efficiency of our proposed $\tau$-matrix based preconditioner. All computations are implemented in a Lenovo desktop with 16GB RAM, AMD Ryzen 5 4600U with Radeon Graphics $@$ 2.10 GHz using MATLAB R2019a.

To effectively demonstrate the performance of our proposed $\tau$-matrix based preconditioner, four numerical methods for solving the scheme (\ref{2D-matrix-form}) are reported as comparisons. For brevity, the unpreconditioned CG method is denoted by `CG'; while `PCG(T)' and `PCG(S)' mean the CG methods with two kinds of circulant-type preconditioners (by substituting $P_{\gamma}$ in (\ref{tau_pre1}) with T. Chan's and Strang's circulant preconditioners constructed from the Toeplitz matrix $A_\gamma$, respectively; see \cite{chan1989toeplitz,chan2007introduction}), and `PCG($\tau$)' represents the CG method with our proposed $\tau$-matrix based preconditioner $P_{\tau}$ defined in (\ref{ptau}). 
%Recall that $h_x$ and $h_y$ are the spatial step-sizes in $x$-direction and $y$-direction, respectively. 
In this section, we set $h_x=h_y=h$ for the related numerical experiments. Then, the discrete $L_2$-norm for the error is defined 
%between the numerical solution and the exact solution 
as
%\begin{equation*}
%{\rm E}(h,\Delta t)=\sqrt{h^2\sum_{i=1}^{N_x}\sum_{j=1}^{N_y}(u^M_{i,j}-\widetilde{u}^M_{i,j})^2},
%\end{equation*}
%where $u_{i,j}^{M}$ and $\widetilde{u}_{i,j}^{M}$ are the numerical solution and exact solution at grid points $(x_i,y_j,t_M)$, respectively. 
\begin{equation*}
{\rm E}(h,\Delta t)=\sqrt{h^2\sum_{i=1}^{N_x}\sum_{j=1}^{N_y}(\psi_{i,j}^M)^2},
\end{equation*}
where $\psi_{i,j}^m$ is defined in (\ref{psi_error}) at grid points $(x_i,y_j,t_M)$. The running CPU time (in seconds) for solving (\ref{2D-matrix-form}) is denoted by `CPU(s)'. The average iterative numbers
%average of iteration numbers 
of the different numerical methods over all time levels are denoted by `Iter', that is,
 \begin{equation*}
\text { Iter }=\frac{1}{M} \sum_{m=1}^{M} \operatorname{Iter}(m).
\end{equation*}
%When the number of iteration of CG method is greater than 1000, the iteration process is terminated and `$\dagger$' as used to represent its CPU time.
The numerical results will not be reported when the number of iterations is more than 1000.

Also, the convergence orders in time and space are measured as follows
\begin{equation*}
R_{\Delta t}=\log _{2}\left(\frac{\operatorname{E}\left(h, \Delta t\right)}{\operatorname{E}\left(\frac{h}{2}, \frac{\Delta t}{2}\right)}\right), \,R_{h}=\log _{2}\left(\frac{\operatorname{E}\left(h, \Delta t\right)}{\operatorname{E}\left(\frac{h}{2}, \Delta t\right)}\right).
\end{equation*}

%In the implementations of 
For the CG and PCG methods, at each time step, we let the initial guess be zero vector, and stop 
%is chosen as the zero vector at each time step and
the iteration when
%the tolerance for the stopping criterion as 
%the tolerance for the stopping criterion is set as
$$
\frac{\|r_k\|}{\|r_0\|}\leq 10^{-10}
$$
with the $k$-th iteration residual vector $r_k$.
%where $r_k$ denotes the residual vector at the $k$-th iteration. 
%Besides, the initial guess is chosen as the zero vector at each time step. 
Furthermore, 
%to reduce the operation cost, 
all the matrix-vector multiplications are computed in ${\mathcal O}(N_xN_y\log(N_xN_y))$ operations via the MATLAB built-in functions {\bf fft}, {\bf ifft} and {\bf dst} for improving the computational efficiency. 
%for reducing the computational cost.

%{\bf Example 1.}
Consider the following 2D space-fractional Fisher's equation in
%with Riesz of order $\alpha,\beta \in(1, 2)$ as follows 
\cite{cheng2019novel,zhu2017numerical,khader2018numerical}
\begin{equation*}\label{example1}
\frac{\partial u(x,y,t)}{\partial t}=K_{\alpha} \frac{\partial^{\alpha} u(x,y,t)}{\partial|x|^{\alpha}}+K_{\beta} \frac{\partial^{\beta} u(x,y,t)}{\partial|y|^{\beta}}+u(1-u)+f(x, y, t), \quad(x, y, t) \in\Omega \times(0, T],
\end{equation*}
%with homogeneous Dirichlet boundary conditions
subject to
\begin{align*}
&u(x, y, 0)=\rho x^{5}(1-x)^{5} y^{5}(1-y)^{5},\quad(x, y) \in \Omega,\\
&u(x, y, t)=0,\quad(x, y, t) \in({\mathbb{R}^2}\textbackslash\Omega)\times(0,T],
%&u(x, y, t)=0,\quad(x, y, t) \in \partial \Omega\times(0,T],
\end{align*}
where $K_{\alpha}=5$, $K_{\beta}=30$, $\Omega=(0,1)\times(0,1)$, $T=1$, and $\rho=10^5$.
%is an arbitrary positive constant.
%\begin{equation*}
%\Omega=(0,1) \times(0,1),~T=1,~K_{\alpha}=5,~K_{\beta}=30,
%\end{equation*}
The exact solution is
\begin{equation*}
u(x, y, t)=\rho e^{-t} x^{5}(1-x)^{5} y^{5}(1-y)^{5},
\end{equation*}
then the nonlinear term is 
%$f(x, y, t, u)$ is 
%consequently determined by the exact solution $u(x, y, t)$, that is,
%with the boundary condition, initial condition, and exact solution are specially defined so that the nonlinear source term is
\begin{equation*}
\begin{aligned}
f(x, y, t,u)=& \rho e^{-t}\Big\{\frac{K_{\alpha} }{2\cos(\frac{\alpha\pi}{2})}  y^{5}(1-y)^{5} \sum_{k=0}^{5}(-1)^{5-k} C_{5}^{k} \frac{\Gamma(11-k)}{\Gamma(11-k-\alpha)} \Big[x^{10-k-\alpha}+(1-x)^{10-k-\alpha}\Big]\\
&+\frac{K_{\beta}}{2\cos(\frac{\beta\pi}{2})} x^{5}(1-x)^{5} \sum_{k=0}^{5}(-1)^{5-k} C_{5}^{k} \frac{\Gamma(11-k)}{\Gamma(11-k-\beta)}
\Big[y^{10-k-\beta}+(1-y)^{10-k-\beta}\Big] \\
&- x^{5}(1-x)^{5} y^{5}(1-y)^{5}\Big\}-u(1-u).
\end{aligned}
\end{equation*}

To validate our proposed scheme (\ref{2D-matrix-form}), the errors and the convergence orders of the scheme (\ref{2D-matrix-form}) in the temporal and spatial directions for different $\alpha$ and $\beta$ is shown in Tables \ref{table1} $\&$ \ref{table2}, respectively. First, the numerical accuracy in time is computed,
%For the temporal direction,
%we take the fixed $h=\tau=\frac{1}{64}$ and
we repeatedly reduce the values of $h$ and $\Delta t$ by half from $h=\Delta t=\frac{1}{64}$ to $h=\Delta t=\frac{1}{1024}$. For the accuracy in space, we take the fixed and sufficiently large %similarly by fixing
$M=10000$ and double the values of $N_x+1$ and $N_y+1$.
%The corresponding errors and convergence orders with respect to $\Delta t$ for different $\alpha$ and $\beta$ are displayed in Table \ref{table1}.
%And then,
%the spatial convergence orders are tested similarly by fixing $M=10000$ and doubling the values of $N_x+1$ and $N_y+1$.
As can be seen from Tables \ref{table1} $\&$ \ref{table2}, the second-order accuracy in time and fourth-order accuracy in space are verified.
%the convergence orders fit the theoretical second-order accuracy in time and fourth-order accuracy in space very well.
%Table \ref{table2} displays the corresponding errors and convergence orders with respect to $h$ for different $\alpha$ and $\beta$. We can observe that the convergence orders displayed in these two tables fit the theoretical second-order temporal accuracy and fourth-order spatial accuracy very well.

In addition, 
%to show the effectiveness of our present method, 
we make a comparison of the fourth-order scheme (\ref{2D-matrix-form}) with the second-order scheme proposed in \cite{zhang2022fast}.
%and the numerical results are listed in Tables \ref{table1} $\&$ \ref{table2}. 
From Tables \ref{table1} $\&$ \ref{table2},
%Tables \ref{table1} and \ref{table2},
one can see that
%the advantages of the fourth-order scheme compared with the second-order scheme in Reference \citenum{zhang2022fast}, that is, 
to obtain a numerical solution for a given error, matrix size of the linear systems needed to be solved for the high-order schemes is much smaller than that of the low-order schemes. Hence, less computational costs and CPU time are consumed, while for the same grid points, the high-order scheme has almost the same computational cost as the low-order schemes but can achieve higher accuracy. For example, if we set the error to 1.0e-7, from Table \ref{table1}, when $(\alpha,\beta)=(1.1,1.2)$, at each time level, the method in \cite{zhang2022fast} needs to solve linear systems with $255\times255=65025$ unknowns, while our fourth-order scheme only needs to solve linear systems with $63\times63=3969$ unknowns. Meanwhile, in Table \ref{table2}, when $(\alpha,\beta)=(1.4,1.5)$, for the fixed $\Delta t=\frac{1}{10000}$ and $h=\frac{1}{64}$, the error derived by our method is $2.1994{\mbox e}$-$8$, which is more accurate than $6.0668{\mbox e}$-$6$ derived by the second order method in \cite{zhang2022fast}.

Furthermore,
%for different $\alpha$, $\beta$, space and time step-sizes,
Table \ref{table3} displays the average iterative numbers 
%average numbers of iterations 
and the CPU time derived by the CG and PCG methods for
%different methods mentioned above for
solving (\ref{2D-matrix-form}) with different $\alpha$ and $\beta$. 
%As shown in Table \ref{table3}, in contrast to the CG method, the average iterative numbers and the CPU time of the PCG method with the aforementioned preconditioners are greatly reduced. 
In Table \ref{table3}, it is shown that, in contrast to the CG method, the average iterative numbers and the CPU time of the PCG method with the aforementioned preconditioners are greatly reduced.
Among the PCG methods, PCG($\tau$) is the most efficient one with an almost constant number of iterations for different $\alpha$ and $\beta$, which 
%is consistent with
is in line with the theoretical analysis
%support the theorem
that the spectra of the preconditioned matrix with the proposed $\tau$-matrix based preconditioner are uniformly bounded, while the average iteration numbers of the circulant-type preconditioners (PCG(T) and PCG(S)) increase as the discritization grid size decreases. In conclusion,
%Therefore, we conclude that 
the proposed $\tau$-matrix based preconditioner outperforms circulant-type preconditioners in terms of both iteration numbers and CPU time.

\begin{table}[!tbp]
\centering
\tabcolsep=5pt
\caption{The comparisons of the errors and convergence orders in time between present method and the method of \cite{zhang2022fast} for the Example at $T=1$.}
\label{table1}
\begin{tabular}{ccccccccc}
\hline
    &\multicolumn{4}{c}{$(\alpha,\beta)=(1.1,1.2)$} &\multicolumn{4}{c}{$(\alpha,\beta)=(1.4,1.5)$}
   \\  \cline{2-5}\cline{6-9}
  $\Delta t=h$ &\multicolumn{2}{c}{Method of \cite{zhang2022fast}}&\multicolumn{2}{c}{Present Method}&\multicolumn{2}{c}{Method of \cite{zhang2022fast}}&\multicolumn{2}{c}{Present Method}\\
   \cline{2-3}\cline{4-5}\cline{6-7}\cline{8-9}
   &${\rm E}(h,\Delta t)$&$R_{\Delta t}$ &${\rm E}(h,\Delta t)$ &$R_{\Delta t}$ &${\rm E}(h,\Delta t)$ &$R_{\Delta t}$ &${\rm E}(h,\Delta t)$ &$R_{\Delta t}$  \\
\hline
         $\frac{1}{64}$    &4.5401e-6        &-      &{\bf2.9347e-7} &-       &5.8729e-6 &-      &2.8832e-7 &-        \\
         $\frac{1}{128}$   &1.1345e-6        &2.0007 &7.4393e-8      &1.9800  &1.4676e-6 &2.0006 &7.3420e-8 &1.9735   \\
         $\frac{1}{256}$   &{\bf2.8360e-7}   &2.0001 &1.8668e-8      &1.9946  &3.6687e-7 &2.0002 &1.8450e-8 &1.9925   \\
         $\frac{1}{512}$   &7.0899e-8        &2.0000 &4.6715e-9      &1.9986  &9.1715e-8 &2.0000 &4.6188e-9 &1.9981    \\
         $\frac{1}{1024}$  &1.7725e-8        &2.0000 &1.1681e-9      &1.9997  &2.2929e-8 &2.0000 &1.1551e-9 &1.9995   \\
\hline
 &\multicolumn{4}{c}{$(\alpha,\beta)=(1.8,1.9)$} &\multicolumn{4}{c}{$(\alpha,\beta)=(1.1,1.9)$}
   \\ \cline{2-5}\cline{6-9}
  $\Delta t=h$ &\multicolumn{2}{c}{Method of \cite{zhang2022fast}}&\multicolumn{2}{c}{Present Method}&\multicolumn{2}{c}{Method of \cite{zhang2022fast}}&\multicolumn{2}{c}{Present Method}\\
   \cline{2-3}\cline{4-5}\cline{6-7}\cline{8-9}
   &${\rm E}(h,\Delta t)$&$R_{\Delta t}$ &${\rm E}(h,\Delta t)$ &$R_{\Delta t}$ &${\rm E}(h,\Delta t)$ &$R_{\Delta t}$ &${\rm E}(h,\Delta t)$ &$R_{\Delta t}$  \\
\hline
         $\frac{1}{64}$    &7.6253e-6 &-      &2.8316e-7 &-       &6.8951e-6 &-      &2.8637e-7 &-        \\
         $\frac{1}{128}$   &1.9319e-6 &1.9808 &7.2500e-8 &1.9656  &1.7514e-6 &1.9771 &7.2801e-8 &1.9759   \\
         $\frac{1}{256}$   &4.8292e-7 &2.0002 &1.8255e-8 &1.9897  &4.3780e-7 &2.0002 &1.8297e-8 &1.9924   \\
         $\frac{1}{512}$   &1.2072e-7 &2.0001 &4.5725e-9 &1.9972  &1.0944e-7 &2.0001 &4.5808e-9 &1.9979    \\
         $\frac{1}{1024}$  &3.0181e-8 &2.0000 &1.1437e-9 &1.9992  &2.7361e-8 &2.0000 &1.1457e-9 &1.9994   \\
\hline
\end{tabular}
\end{table}

\begin{table}[t]%
\centering
\caption{The comparisons of the errors and convergence orders in space between present method and the method of \cite{zhang2022fast} for the Example at $T=1$ with $\Delta t=\frac{1}{10000}$.} 
%for different $\alpha$ and $\beta$.}
\label{table2}
\begin{tabular}{ccccccccc}%
 \hline
 &\multicolumn{4}{c}{$(\alpha,\beta)=(1.1,1.2)$} &\multicolumn{4}{c}{$(\alpha,\beta)=(1.4,1.5)$}
   \\  \cline{2-5}\cline{6-9}
  $h$ &\multicolumn{2}{c}{Method of \cite{zhang2022fast}}&\multicolumn{2}{c}{Present Method}&\multicolumn{2}{c}{Method of \cite{zhang2022fast}}&\multicolumn{2}{c}{Present Method}\\
   \cline{2-3}\cline{4-5}\cline{6-7}\cline{8-9}
   &${\rm E}(h,\Delta t)$&$R_{h}$ &${\rm E}(h,\Delta t)$&$R_{h}$ &${\rm E}(h,\Delta t)$&$R_{h}$ &${\rm E}(h,\Delta t)$&$R_{h}$  \\
\hline
         $\frac{1}{8}$    &3.1676e-4        &-      &5.6320e-5          &-       &4.0962e-4      &-      &7.7184e-5      &-        \\
         $\frac{1}{16}$   &7.6527e-5        &2.0494 &3.9238e-6          &3.8433  &9.8302e-5      &2.0590 &5.3423e-6      &3.8528   \\
         $\frac{1}{32}$   &1.8969e-5        &2.0123 &2.5467e-7          &3.9455  &2.4326e-5      &2.0147 &3.4706e-7      &3.9442   \\
         $\frac{1}{64}$   &4.7326e-6        &2.0029 &1.6123e-8          &3.9814  &{\bf6.0668e-6} &2.0035 &{\bf2.1994e-8} &3.9800    \\
         $\frac{1}{128}$  &1.1826e-6        &2.0007 &1.0077e-9          &3.9999  &1.5158e-6      &2.0008 &1.3778e-9      &3.9967   \\
\hline
&\multicolumn{4}{c}{$(\alpha,\beta)=(1.8,1.9)$} &\multicolumn{4}{c}{$(\alpha,\beta)=(1.1,1.9)$}
   \\  \cline{2-5}\cline{6-9}
  $h$ &\multicolumn{2}{c}{Method of \cite{zhang2022fast}}&\multicolumn{2}{c}{Present Method}&\multicolumn{2}{c}{Method of \cite{zhang2022fast}}&\multicolumn{2}{c}{Present Method}\\
   \cline{2-3}\cline{4-5}\cline{6-7}\cline{8-9}
   &${\rm E}(h,\Delta t)$&$R_{h}$ &${\rm E}(h,\Delta t)$&$R_{h}$ &${\rm E}(h,\Delta t)$&$R_{h}$ &${\rm E}(h,\Delta t)$&$R_{h}$  \\ 
\hline
         $\frac{1}{8}$    &5.4258e-4 &-      &1.1096e-4 &-       &4.9713e-4 &-      &1.0421e-4 &-        \\
         $\frac{1}{16}$   &1.2886e-4 &2.0740 &7.5501e-6 &3.8774  &1.1708e-4 &2.0861 &6.7998e-6 &3.9378   \\
         $\frac{1}{32}$   &3.1806e-5 &2.0184 &4.8974e-7 &3.9464  &2.8860e-5 &2.0204 &4.3457e-7 &3.9678   \\
         $\frac{1}{64}$   &7.9263e-6 &2.0046 &3.1020e-8 &3.9808  &7.1899e-6 &2.0050 &2.7335e-8 &3.9908    \\
         $\frac{1}{128}$  &1.9800e-6 &2.0011 &1.9452e-9 &3.9952  &1.7959e-6 &2.0013 &1.7086e-9 &3.9999   \\
\hline
\end{tabular}
\end{table}

\begin{table}[t]%
\centering
\caption{The comparisons between CG and PCG methods with the aforementioned preconditioners for solving the Example at $T=1$ with $N={N_x+1}={N_y+1}$.} 
%for different $\alpha$ and $\beta$.}
\label{table3}
\begin{tabular}{ccccccccccccccccccc}
\hline
\multirow{2}{*}{$(\alpha,\beta)$} & \multirow{2}{*}{$M$}  & \multirow{2}{*}{$N$}  &  \multicolumn{2}{c}{CG} &  \multicolumn{2}{c}{PCG(T)} & \multicolumn{2}{c}{PCG(S)}  & \multicolumn{2}{c}{PCG($\tau$)} \\  \cline{4-5} \cline{6-7} \cline{8-9}\cline{10-11}
     &    &  &CPU(s)  &Iter &CPU(s)  &Iter &CPU(s)  &Iter &CPU(s)  &Iter  \\
\hline
\multirow{4}{*}{(1.1, 1.2)}
         &$2^{3}$   &$2^{8}$   &10.83         &192.00    &2.36      &32.00    &1.76      &22.00    &1.42    &10.00  \\
         &$2^{4}$   &$2^{9}$   &154.87        &285.00    &27.43     &37.00    &19.70     &25.50    &14.17   &11.00  \\
         &$2^{5}$   &$2^{10}$  &1660.70       &407.00    &221.98    &41.00    &151.00    &26.94    &109.66  &11.00   \\
         &$2^{6}$   &$2^{11}$  &20563.14      &556.98    &2214.01   &44.00    &1545.68   &29.91    &893.89  &11.00   \\
\hline
\multirow{4}{*}{(1.4, 1.5)}
         &$2^{3}$   &$2^{8}$   &13.66          &255.00    &3.60      &50.88    &2.32      &30.75    &1.38    &10.00  \\
         &$2^{4}$   &$2^{9}$   &271.53         &511.00    &46.62     &65.38    &26.84     &37.50    &11.78   &10.00  \\
         &$2^{5}$   &$2^{10}$  &3660.02        &904.00    &457.12    &85.03    &230.48    &42.94    &91.48   &10.00    \\
         &$2^{6}$   &$2^{11}$  &$\dagger$      &$>$1000   &5282.73   &108.20   &2407.43   &48.00    &820.12  &10.00    \\
\hline
\multirow{4}{*}{(1.8, 1.9)}
         &$2^{3}$   &$2^{8}$   &13.60          &239.88    &6.65      &97.50     &3.01     &42.63    &1.03    &7.00  \\
         &$2^{4}$   &$2^{9}$   &268.98         &492.69    &107.37    &156.44    &37.28    &53.25    &8.70    &7.00  \\
         &$2^{5}$   &$2^{10}$  &$\dagger$        &$>$1000   &1299.36   &254.25    &358.68   &68.25    &74.00   &8.00    \\
         &$2^{6}$   &$2^{11}$  &$\dagger$      &$>$1000   &19917.05  &412.91    &4274.96  &86.33    &680.23  &8.00   \\   
\hline         
 \multirow{4}{*}{(1.1, 1.9)}
         &$2^{3}$   &$2^{8}$   &13.80         &249.50    &7.86       &113.00    &3.62     &51.88    &1.14    &8.00  \\
         &$2^{4}$   &$2^{9}$   &272.60        &495.56    &113.81     &169.00    &45.46    &65.81    &9.80    &8.00 \\
         &$2^{5}$   &$2^{10}$  &4150.15       &982.81    &1249.75    &238.31    &428.63   &82.63    &81.89   &9.00  \\
         &$2^{6}$   &$2^{11}$  &$\dagger$     &$>$1000   &16404.45   &332.00    &4967.90  &100.58   &747.87  &9.00   \\           
\hline
\end{tabular}
\end{table}
\section{Concluding remarks}

 In this paper, we have proposed a novel fourth-order symmetric scheme for the 2D RSFNRDEs, and analyzed the stability and convergence of the proposed scheme under the Lipschitz condition for the nonlinear term. Moreover, a $\tau$-matrix based preconditioner is also developed to efficiently solve the symmetric linear system arising from the 2D RSFNRDEs. In theory, all eigenvalues of 
 the preconditioned matrices are proved to be strictly greater than $3/8$ and less than $2$,
% the spectra of the preconditioned matrices are proved to be uniformly bounded in the open interval $(3/8,2)$, 
 which can guarantee that the PCG method converges linearly and the number of iterations is independent of the mesh sizes.
 %linearly within the iterative number independent of the matrix size. 
 Numerical examples are given to validate the correctness of the scheme and the effectiveness of the proposed preconditioner. As a basis for comparison, we compare the proposed $\tau$-matrix based preconditioner with two commonly used circulant preconditioners. Numerical results show that the proposed $\tau$-matrix based preconditioner outperforms circulant-type preconditioners in terms of less CPU time and iteration numbers. It is worth noting that our work is the first attempt to develop a preconditioned iterative solver in the literature with a mesh-independent convergence rate for the linearized high-order scheme. In our future work, the preconditioning strategy for the higher-dimensional case will be studied.

\end{document}